\documentclass[aos,preprint]{imsart}

\RequirePackage[numbers]{natbib}
\usepackage[ansinew]{inputenc}
\usepackage{amssymb}
\usepackage{amsmath}
\usepackage{amsthm}
\usepackage{bm}
\usepackage{bbm}
\usepackage{mathrsfs}
\usepackage{enumerate}

\startlocaldefs
\input{math.sty}

\endlocaldefs

\begin{document}

\begin{frontmatter}

\title{Nonparametric Bayesian posterior contraction rates for discretely observed scalar diffusions}
\runtitle{Posterior contraction rates for diffusions}


\author{\fnms{Richard} \snm{Nickl}\ead[label=e1]{r.nickl@statslab.cam.ac.uk}}
\and
\author{\fnms{Jakob} \snm{S\"ohl}\ead[label=e2]{j.soehl@statslab.cam.ac.uk}}
\runauthor{R. Nickl and J. S\"ohl}

\affiliation{University of Cambridge}
\address{Statistical Laboratory\\
Department of
Pure Mathematics and Mathematical Statistics\\
University of Cambridge\\
CB3 0WB, Cambridge, UK\\
\printead{e1}\\
\printead{e2}}

\begin{abstract}
We consider nonparametric Bayesian inference in a reflected diffusion model 
$dX_t = b (X_t)dt + \sigma(X_t) dW_t,$ with discretely sampled observations 
$X_0, X_\Delta, \dots, X_{n\Delta}$. We analyse the nonlinear inverse problem 
corresponding to the `low frequency sampling' regime where $\Delta>0$ is fixed 
and $n \to \infty$. A general theorem is proved that gives conditions for prior 
distributions $\Pi$ on the diffusion coefficient $\sigma$ and the drift function 
$b$ that ensure minimax optimal contraction rates of the posterior distribution 
over H\"older-Sobolev smoothness classes. These 
conditions are verified for natural examples of nonparametric random wavelet series  
priors.  For the proofs we derive 
new concentration inequalities for empirical processes arising from discretely 
observed diffusions that are of independent interest.
\end{abstract}



\end{frontmatter}

\section{Introduction}

Many fundamental models for dynamic stochastic 
phenomena in continuous time are based on the concept of a \textit{diffusion}, whose evolution is
modelled mathematically by $$dX_t = b (X_t)dt + \sigma(X_t) dW_t,~~ t\ge 0,$$ 
where $W_t$ is a standard Brownian motion. Diffusions occur naturally in the 
physical and biological sciences, in economics and elsewhere, and their deep 
relationship to stochastic and partial differential equations makes them a 
central object of study in modern mathematics. Various specifications of the 
\textit{drift function} $b$ and the \textit{diffusion coefficient} $\sigma$ lead 
to a flexible class of random continuous motions. In scientific 
applications, a key challenge is to recover the parameters $b, \sigma$ from some 
form of observations of the diffusion. Unless specific knowledge is available, 
the resulting statistical models for the parameters $\sigma, b$ and the 
probability laws $\mathbb P_{\sigma b}$ of the Markov process $(X_t: t \ge 0)$ 
are naturally \textit{infinite-dimensional} (=`nonparametric'). 

Statistical observations in the real world usually are collected in a discrete 
fashion, say in form of observed increments $X_0, X_\Delta, \dots, X_{n 
\Delta}$ 
of the diffusion, where $1/\Delta$ is the sampling frequency. We are interested 
in the possibly most realistic scenario where $\Delta>0$ is fixed and more 
information accrues in form of an increasing sampling horizon $n \Delta \to 
\infty$. As revealed in the seminal paper by Gobet, Hoffmann and 
Rei\ss~\cite{gobetETAL2004}, this `low 
frequency' sampling regime implies that inference on $(\sigma, b)$ constitutes 
a nonlinear nonparametric inverse problem, and the authors solve this problem in 
a minimax way by a delicate estimation technique based on ideas from 
spectral theory. 

Alternative methodology for nonparametric inference in diffusion models has been 
put forward 
recently, notably of a Bayesian flavour, see Roberts and Stramer
\cite{RobertsStramer2001}, Papaspiliopoulos et al.
\cite{PapaspiliopoulosPokernRobertsStuart2012}, Pokern et al.
\cite{PokernStuartvanZanten2013}, 
van der Meulen et al. \cite{vanderMeulenSchauervanZanten2014}, van Waaij 
and van Zanten
\cite{WaaijvanZanten2015} and references therein. 
While such Bayesian methods are attractive in applications 
\cite{GolightlyWilkinson2005}, \cite{Stuart2010}, 
\cite{vanZanten2013}, particularly since they provide associated 
uncertainty quantification procedures (`credible regions'), our understanding 
of their frequentist sampling performance is extremely limited. This is 
particularly so in the `low frequency' regime when $\Delta>0$ is thought to be 
fixed: the only references we are aware of are the consistency results in 
\cite{vanderMeulenvanZanten2013}, \cite{GugushviliSpreij2014},
\cite{KoskelaSpanoJenkins2015}, which only hold under the very restrictive 
assumption that $\sigma$ is 
constant and known, and only in a weak topology. As pointed out by 
Stuart \cite{Stuart2010} and van Zanten \cite{vanZanten2013}, 
obtaining theoretical performance guarantees for Bayesian algorithms in nonlinear inverse problems is, however, of key importance if such methods are to be 
used in scientific applications. Only very few rigorous results are currently available.

In this paper we give the first proof of the fact that nonparametric prior 
distributions on the diffusion parameters $(\sigma, b)$ give rise to posterior distributions that 
contract at the (minimax) frequentist optimal convergence rates over natural 
regularity classes, in the low frequency sampling regime. This is achieved by using the generic `testing 
approach' 
introduced in the landmark paper Ghosal et 
al. \cite{GhosalGhoshvanderVaart2000}, see also 
\cite{GhosalvanderVaart2007} and \cite{vanderVaartvanZanten2008}
-- but the adaptation to the diffusion case requires the resolution of two 
major mathematical obstacles to obtain satisfactory results:

$\bullet$ The \textit{`small ball probability conditions'} need to reflect the 
inverse problem nature of the discrete sampling scheme, and the resulting 
perturbation of the `information theoretic distance' (KL-divergence) associated 
to the statistical experiment needs to be precisely quantified. For linear 
inverse problems this has already been noted in the paper by 
Ray \cite{Ray2013}. 
In the nonlinear diffusion setting here, however, the situation is much more 
complicated, and requires a fine analysis of the inverse operator of the 
infinitesimal generator of the diffusion (which could be viewed as the 
linearisation of the non-linear inverse operator). In the case of high 
frequency 
($\Delta \to 0$) or continuous observations, small ball probabilities may be 
computed for an information distance closely related to the $L^2$-distance on 
$b$ and $\sigma$ (see \cite{vanderMeulenvanderVaartvanZanten2006}), but to 
obtain optimal results in the low frequency 
setting one has to show that instead small ball probabilities may be computed 
in 
a weaker norm -- precisely, as we show, in a certain negative order Besov norm. 
Simultaneously one has to ensure that the invariant measure $\mu$ is correctly 
modelled by the (induced) prior too -- note that the smoothness degree of $\mu$ 
is generally not identified by the regularity of $(\sigma, b)$. This last fact 
also should guide practitioners who often devise priors for $\sigma$ and $b$ 
without paying attention to the implied model for the invariant measure.

$\bullet$ The \textit{construction of frequentist tests} with sufficiently good 
exponential error bounds for type two errors in a large enough support set of 
the prior in \cite{GhosalGhoshvanderVaart2000, GhosalvanderVaart2007} relies on 
properties of the 
likelihood ratio test and the associated Hellinger distance between 
experiments. 
In the setting of diffusions this approach appears difficult to implement -- 
instead we use 
the `concentration of measure' approach of Gin\'e and Nickl 
\cite{GineNickl2011} to the 
construction of such tests. To do this we prove a Bernstein-type inequality for 
empirical processes driven by discretely sampled diffusions, relying on 
work of Adamczak \cite{Adamczak2008}, and use it to derive sharp concentration 
bounds 
for the estimators (and resulting plug-in tests) put forward in Gobet et al.
\cite{gobetETAL2004}. These concentration results, which are of independent 
interest, are derived in Section \ref{diffdet}.

\smallskip

We demonstrate that our general conditions are verified for natural 
nonparametric priors on $(\sigma, b)$. It is convenient to give a hierarchical 
prior specification that first models the inverse diffusion coefficient 
$\sigma^{-2}$, and then, conditional on $\sigma^2$, the drift function $b$, 
which explicitly generates a prior for the invariant measure $\mu$ too. The 
individual prior choices are quite flexible and allow for general random series priors, as we show, with one technically vital restriction that 
they are constrained to a fixed regularity class that ensures sufficient 
smoothness of the individual parameters. This is necessary to deduce various 
probabilistic properties of the diffusion that our proofs rely on. In 
particular, following \cite{gobetETAL2004}, we rely on the assumption that the 
diffusion considered is a reflected one, and hence lives in a compact interval 
of $\mathbb R$. This corresponds to the usual von Neumann boundary conditions 
required for the infinitesimal generator $L$ to be injective and to have a 
discrete spectrum. As a consequence, to cope with these boundary conditions, we 
model $b$ and $\sigma$ only in the interior of the given interval (this also 
has 
some deeper mathematical reasons since the second eigenfunction of $L$ 
identifies 
$b,\sigma$ only in the interior of the domain, and since our approach to 
construct tests is based on first estimating this eigenfunction). Again, these 
are concessions to the mathematical intricacies of the problem at hand, and 
further hard work will be required to alleviate those. We discuss possible 
extensions and limitations of our approach in Subsection \ref{outlook} below.

\section{Main results}

\subsection{A nonparametric model for diffusions on $[0,1]$}

Consider a scalar diffusion process $(X_t: t \ge 0)$ on $[0,1]$ starting at 
$X_0=x_0$, and whose evolution is described by the stochastic differential 
equation (SDE)
\begin{equation} \label{diffm}
 \d X_t=b(X_t)\d t+\sigma(X_t)\d W_t, ~~~ t \ge 0,
\end{equation}
where the process is reflected at the boundary points $\{0,1\}$ (for a precise 
definition see Section \ref{diffdet} below). For the pair $\vartheta = 
(\sigma,b) \in C([0,1]) \times C([0,1])$ we maintain the following model
\begin{align*}
 \Theta:=\Big\{&\vartheta = (\sigma,b): b(0)=b(1)=\sigma'(0)=\sigma'(1)=0,
\text{ $b', \sigma', \sigma''$ exist,} \\ 
&~~ \max\left(\|b\|_\infty, \|b'\|_\infty, 
\|\sigma\|_\infty, 
\|\sigma'\|_\infty, \|\sigma''\|_\infty \right)\le D, \inf_{x \in 
[0,1]}\sigma^2(x) \ge d \Big\}
\end{align*}
where $D,d$ are arbitrary fixed positive constants. For $(\sigma, b) \in 
\Theta$ 
the SDE (\ref{diffm}) has a pathwise solution described by the Markov process 
$(X_t: t \ge 0)$ with invariant measure $\mu=\mu_{\sigma b}$, whose law we denote by $\mathbb 
P_{\sigma b}$ whenever $X_0 \sim \mu$. The observation scheme considered is such that increments $X_0,X_{\Delta},\dots,X_{n\Delta}$ are sampled 
at distance $\Delta>0$, and we study statistical inference on 
$\sigma$ and $b$ when $n$ tends to infinity as $\Delta>0$ remains fixed (the 
`low frequency sampling' regime). Thus the $X_{i\Delta}$'s form an ergodic 
Markov chain and we write $p_{\sigma b}(\Delta, x,y)$ for 
the associated transition probability density functions with respect to Lebesgue measure $\d y$. We shall also - in abuse of notation - write $\mu$ both for the invariant measure and its density.

Let $\Pi$ be a (prior) probability distribution on some $\sigma$-field $\mathcal S$ of 
subsets 
of $\Theta$, and given $(\sigma,b) \sim \Pi$ assume that the law of $(X_t: t 
\ge 
0)|(\sigma, b)$ is described by the diffusion (\ref{diffm}) started in the invariant measure $\mu_{\sigma b}$. If the mapping 
$(\sigma, b) \mapsto p_{\sigma b}(\Delta, x,y)$ is $\mathcal S$-$\mathcal 
B_\mathbb R$ measurable for all $x,y$, then by standard arguments (as in 
Chapter 7.3 in \cite{GineNickl2015}) the posterior distribution given the discrete 
sample from the diffusion is 
\begin{equation} \label{post}
(\sigma,b)|X_0, X_\Delta, \dots, X_{n\Delta} \sim 
\frac{\mu_{\sigma b}(X_0) \Pi_{i=1}^np_{\sigma b}(\Delta, X_{(i-1)\Delta}, X_{i 
\Delta}) d\Pi((\sigma, b))}{\int_\Theta \mu_{\sigma b}(X_0)  
\Pi_{i=1}^np_{\sigma b}(\Delta, X_{(i-1)\Delta}, X_{i \Delta}) d\Pi((\sigma, 
b))}.
\end{equation}
 
 We wish to devise natural conditions on the prior $\Pi$ that imply that 
the posterior distribution contracts at the optimal convergence rate $\delta_n$ in some distance function $d$ about any 
fixed `true' parameter pair $\vartheta_0=(\sigma_0, b_0) \in \Theta$. More precisely, we wish to prove that, as $n \to \infty$, $$\Pi(\vartheta: d(\vartheta_0, \vartheta)>\delta_n|X_0, X_\Delta, \dots, X_{n\Delta}) \to 0,~  \text{ in } \mathbb P_{\sigma_0 b_0} \text{- probability},$$ under the `frequentist' assumption that $(X_t: t \ge 0) \sim \mathbb P_{\sigma_0 b_0}$. The rate $\delta_n$ will depend on regularity properties of $(\sigma_0, b_0)$ that we describe now.

\subsection{Contraction Theorem}

In \cite{gobetETAL2004} it was shown that the frequentist minimax rates for 
estimating the parameter $(\sigma, b)$ in $L^2([A,B])$-loss ($0<A<B<1$) are 
given by 
\begin{equation} \label{minrat}
n^{-s/(2s+3)}~~\text{for } \sigma^2 ~~\text{and } n^{-(s-1)/(2s+3)}~~\text{for 
} b
\end{equation}
whenever $(\sigma, b) \in \Theta_s$, where the regularity classes $\Theta_s \subset \Theta$ are defined as
\begin{align*}
 \Theta_s:=\left\{\vartheta=(\sigma,b)\in\Theta:\|\sigma\|_{H^s}\le D, \|b\|_{H^{s-1}}\le D\right\},~~s\ge 2,
\end{align*}
with $H^s$ the usual $L^2$-Sobolev space over $[0,1]$. This particular coupling of the regularities $s$ of $\sigma$ and $s-1$ of $b$ is natural, see Remark \ref{coupling} below.

The above rates reflect the recovery complexity of an ill-posed problem of 
order 
one and two, respectively. For the Bayesian posterior distribution to have good 
frequentist properties it is well known (\cite{GhosalGhoshvanderVaart2000}) 
that 
the prior should charge small neighbourhoods of the `true pair' $(\sigma_0, 
b_0)$ with sufficient probability, where `neighbourhood' is understood with 
respect to the information distance induced by the observations. As noted by 
\cite{Ray2013} for linear inverse problems $$Y=Af +\epsilon,~~~A: L^2 \to L^2 \textit{ linear},$$ with unknown parameter $f \in L^2$ and Gaussian white noise 
$\epsilon$, a key point is to take advantage 
of the fact that the usual information distance $\|f\|_{L^2}$ (when $A=Id$) is transformed to 
$\|Af\|_{L^2}$, which typically corresponds to a negative (or dual) Sobolev 
norm $\|f\|_{H^{-w}}= \|f\|_{(H^w)^*}$ induced by the eigen-basis of $A$, and 
where $w$ denotes the level of ill-posedness.

One of the key contributions of this paper is to produce a similar result in 
our non-linear and non-Gaussian setting of discretely sampling a diffusion. To obtain sharp results 
the Hilbert scale of Sobolev norms will have to be replaced by more flexible 
Besov norms. To this end, for $s>0$ denote by $(B^s_{1\infty})^*$ the dual 
space of the Besov space  $B^s_{1\infty}=B^s_{1\infty}([0,1])$, equipped with the usual dual norm (see (\ref{dualnorm}) below). We refer to \cite{M92, triebel2010} or Chapter~4.3 in 
\cite{GineNickl2015} 
for the usual definitions and basic properties of Besov spaces. We further note 
that any prior distribution on $(\sigma, b)$ induces a prior distribution on 
the invariant measure $\mu$ of the diffusion (see (\ref{invid}) below), 
and the following result implicitly requires this induced prior to correctly 
model the parameter $\mu$ too. See Remark \ref{coupling} for discussion.

\smallskip

In what follows, for two sequences $(a_m: m \in \mathbb N), (b_m: m \in \mathbb N)$, we write $a_m \lesssim b_m$ whenever $a_m \le C b_m$ for all $m \in \mathbb N$ and some fixed constant $C>0$, and we write $a_m \simeq b_m$ whenever both $a_m \lesssim b_m$ and $b_m \lesssim a_m$ hold.

\begin{theorem}\label{thm:main}
Let $\Pi=\Pi_n$ be a sequence of prior distributions on $\Theta$, suppose that 
$X_0,\dots ,X_{n\Delta}$ are discrete observations of a 
diffusion process (\ref{diffm}) started in the 
stationary distribution $\mu$, and let $\Pi(\cdot|X_0,\dots,X_{n\Delta})$ be 
the resulting posterior distribution (\ref{post}) on $\Theta$.

Assume $\Pi$ satisfies for some 
$(\sigma_0, b_0) \in \Theta_s$ and $\mu_0 \in L^2$, that \\
 \emph{(i)} $\Pi(\Theta_s)=1$ for some $s \ge 2$ (and for some $d>0, D>0$),\\
  \emph{(ii)} there exists a constant $C>0$ and a sequence $\eps_n$ satisfying 
$$n^{-(s+1)/(2s+3)} \lesssim \eps_n \lesssim n^{-3/8} (\log n)^{-1/2}$$ such that for all $n$ large enough
\begin{align}\label{eq:smallball}
 \Pi\left((\sigma,b)\in\Theta: 
\|\mu-\mu_0\|_{L^2([0,1])}+
\|\sigma^{-2}-\sigma_0^{-2}\|_{(B^1_{1\infty})^*}
+\|b-b_0\|_{(B^2_{1\infty})^*}<\eps_n\right) \notag \\ \ge e^{-Cn\eps_n^2}.
\end{align}
Then if $(X_t: t \ge 0) \sim \mathbb P_{\sigma_0 b_0}$ where $X_0 \sim 
\mu_0$, for true parameters $(\sigma_0, b_0)$ with associated 
invariant measure $\mu_0=\mu_{\sigma_0 b_0}$, the posterior distribution contracts about 
$(\sigma^2_0, b_0)$ in $L^2\equiv L^2([A,B])$ for any $0<A<B<1$, at rates $$\delta_n\equiv 
n\varepsilon_n^3 ~~and~ 
\delta_n' \equiv n^2 \varepsilon_n^5;$$ that is, for some fixed constant $M$, 
as 
$n \to \infty$ and in $\PP_{\sigma_0 b_0}$-probability,
\begin{align*}
\Pi\left((\sigma,b):\|\sigma^2-\sigma_0^2\|_{L^2}>M\delta_n\text{ 
or }\|b-b_0\|_{L^2}>M\delta_n'|X_0,\dots,X_{n\Delta}
\right) \to0.
\end{align*}
\end{theorem}

\begin{remark} \label{opt} \normalfont \textit{Optimal rates.}
The optimal choice of $\varepsilon_n$ is of the order $$\varepsilon_n \simeq n^{-(s+1)/(2s+3)}$$ in which case 
$\varepsilon_n = O(n^{-3/8} (\log n)^{-1/2})$ is always satisfied since $s \ge 2$, and the 
resulting contraction rates are of the desired minimax order: $$\delta_n \simeq 
n^{-s/(2s+3)}, ~~~\delta_n' \simeq n^{-(s-1)/(2s+3)}.$$ 
\end{remark}

\subsection{Examples of Prior distributions}

\subsubsection{Some preliminaries on function spaces}

We now show how Theorem~\ref{thm:main} applies to some concrete prior 
distributions. To do so we need to define the following H\"older-type function 
spaces $\mathcal C^t$:
\begin{definition}
For $t>0$ and $\lfloor t\rfloor$ the largest integer $k \le t$ 
we define
\begin{align*}
\mathcal C^t([0,1])&:=\left\{f\in C([0,1]): |||f|||_{\mathcal C^t} < \infty 
\right\},\quad\text{where}\\
|||f|||_{\mathcal C^t}&:=\sum_{k=0}^{\lfloor t \rfloor}\|D^k f\|_\infty
+\sup_{0 \le x < x+h \le 1}\frac{|D^{\lfloor 
t\rfloor}f(x+h)-D^{\lfloor t\rfloor}f(x)|}{ h^{t-\lfloor 
t\rfloor}\log(1/h)^{-2} 
}.
\end{align*}
\end{definition}
The additional logarithmic factor in the H\"older condition is convenient as 
then $$f \in \mathcal C^t  \Rightarrow f \in H^t \cap B^t_{\infty 1}$$ follows, 
which allows to combine knowledge of spectral properties of the diffusion 
expressed in terms of Sobolev $H^t$-norms with wavelet characterisations of 
H\"older and Besov spaces. Note that the continuous imbeddings of $\mathcal 
C^t$ 
into $H^t$ and into $B^t_{\infty 1}$ follow easily from wavelet 
characterisations of the norms of these spaces (see \cite{M92} or \cite{GineNickl2015}, 
Chapter 4.3, where we note that $H^t=B^t_{22}$), and these also imply that 
$B^t_{\infty 1}, t \in \mathbb N,$ is continuously imbedded into the classical 
spaces $C^t$ of $t$-times continuously differentiable functions.

In fact we will work with the equivalent wavelet norm of $\mathcal C^t$ given by
\begin{equation} \label{wavnorm}
\|f\|_{\mathcal C^t} \equiv \|f\|_{\mathcal C^{t,2}},~~~\|f\|_{\mathcal 
C^{t,\gamma}} = \sup_{l,k} 2^{l(t+1/2)} l^\gamma|\langle f, \psi_{lk} \rangle|,
\end{equation}
where $\{\psi_{lk}: k =0, \dots 2^l-1, l \ge J_0-1\}, J_0 \ge 2, J_0 \in \mathbb N,$ 
is a sufficiently regular boundary-adapted Daubechies wavelet basis of 
$L^2([0,1])$ that also generates $H^t, \mathcal C^t$ as well as the Besov spaces 
$B^t_{pq}$ (see Chapter 4.3.5 in \cite{GineNickl2015}). 

We shall assume that the true pair $(\sigma_0, b_0)$ lies in $\Theta_s \cap (
\mathcal C^s \times \mathcal C^{s-1})$. Moreover, to avoid tedious technicalities about boundary 
conditions, we assume that $b_0$ is supported in the interior of $[0,1]$, and 
that $\log \sigma_0^{-2}$ and $\log \mu_0$ have expansions into wavelet series 
supported in the interior of $[0,1]$. Such functions can be modelled by 
infinite 
Daubechies wavelet series $\psi_{lk}$ that are supported in a given fixed 
interval $[A,B]$. The minimax estimation rates over functions satisfying these 
constraints are the same as those in (\ref{minrat}) up to $\log n$ factors (see Remark \ref{minrem} below), and the minor loss in generality comes at the gain of 
substantial technical simplifications. Thus we adopt the following 
condition, formulated in terms of $\mu_0$ and $\sigma_0$ (implicitly 
defining $b_0$).

\begin{assumption}\label{decayassumpt}
For $0<A<B<1$ given, let $\mathcal I$ be the maximal set of double indices 
$(l,k)$ such that the Daubechies wavelet functions $\psi_{lk}, (l,k) \in 
\mathcal I$, are all supported in $[A,B]$.

We assume that the invariant density $\mu_0 \in \mathcal C^{s+1}$ is of the form
$$\log \mu_0(x)=\sum_{l,k \in \mathcal I}\beta_{lk} \psi_ {lk} (x),~~x \in 
[0,1],$$ 
with $2^{l(s+3/2)}l^2|\beta_{lk}|\le B$ for some $B>0$.

We further assume that the diffusion coefficient $\sigma_0 \in \mathcal C^s$ 
has 
the form 
$$\log \sigma_0^{-2}(x)=\sum_{l, k \in \mathcal I} \tau_{lk} \psi_ {lk} (x), 
~~x 
\in [0,1],$$ where
$2^{l(s+1/2)}l^2|\tau_{lk}|\le B$. 
\end{assumption}

Note that the assumptions ensure that $\sigma_0^{-2}$ and $\mu_0$ are bounded and bounded 
away from zero on $[0,1]$, and that as a consequence, so is $\sigma_0^2 \in 
\mathcal C^s$. It also implies that $2b_0 =  (\sigma_0^2 \mu_0)'/\mu_0$ is 
contained in $\mathcal C^{s-1}$ and supported in $[A,B]$ (since $\sigma_0^2 
\mu_0$ is constant outside of that interval). 

\begin{remark} \label{minrem} \normalfont \textit{[Minimax rates over $\mathcal C^s$-classes.]}
Assumption \ref{decayassumpt} is restricting $(\sigma_0, b_0)$ beyond having to 
lie in $\Theta_s$. The lower bound proofs in \cite{gobetETAL2004} imply that 
these further restrictions do not change the minimax rates, except for a $(\log 
n)^\gamma, \gamma>0,$ factor induced by the weighting with the factor $l^2$ in the wavelet 
norm. Note that the lower bound in \cite{gobetETAL2004} is also based on 
wavelets that are supported in the interior of $[0,1]$, and works with constant 
invariant density $\mu_0=1 \in \mathcal C^{s+1}$, which means that $2b$ just equals $(\sigma^2)'$.
\end{remark}

\subsubsection{Random wavelet series prior}

We consider the following hierarchical prior specification. Let $s \ge 2$. Wavelet coefficients will be constructed from random variables drawn i.i.d.~from probability density 
\begin{equation}\label{density}
\phi:[-\tilde B,\tilde B] \to [0,\infty), ~ B\le \tilde B < \infty, ~\inf_{x \in [-B,B]}\phi (x)\ge\zeta, ~\zeta>0.
\end{equation}
This in particular includes the cases where $\phi$ is the density of a uniform $U[-B,B]$ random variable (so that $\zeta=1/(2B)$), or the case where $\phi$ equals the density of the truncated normal distribution given by $\phi(x)\simeq e^{-x^2/2}$ for $x\in[-B,B]$ and $\phi(x)=0$ otherwise, (so that $\zeta \ge ({2\pi})^{-1/2}e^{-B^2/2}$).

\smallskip

We first model $\log \sigma^{-2}$ as a wavelet series, that is,
\begin{align*}
\sigma^{-2}(x)&= \exp\left\{\sum_{l,k\in\mathcal I, l\le L_n}2^{-l(s+1/2)}l^{-2}
u_ { lk } \psi_ {lk} (x) \right\},~~x \in [0,1],
\end{align*}
where $u_{lk}$ are drawn i.i.d.~from density $\phi$ satisfying (\ref{density}). Here we can take $L_n = \infty$ (so that the prior is independent of $n$) but in our result below we also allow for $L_n$ to equal a sequence of integers diverging with $n$. 

\smallskip

Conditional on $\sigma$ we use the identity
\begin{equation} \label{hierid}
2b = \frac{(\sigma^2 \mu)'}{\mu} = (\sigma^2)' + \sigma^2 (\log \mu)',
\end{equation}
and the law of $b|\sigma^2$ is modelled by taking a wavelet prior $$H(x)=\sum_{l,k\in\mathcal I, l\le \bar L_n}2^{-l(s+3/2)}l^{-2}\bar u_{lk} \psi_ {lk} (x),~x \in [0,1], ~\bar L_n \in \mathbb N \cup \{\infty\},$$ and 
the resulting prior $e^H/\int e^H$ on the parameter $\mu$, where the $\bar 
u_{lk}$ are drawn i.i.d.~from density $\bar \phi$ satisfying (\ref{density}), independent of the $u_{lk}$'s from above. Concretely $$b|\sigma^2 = ((\sigma^2)' + \sigma^2 H')/2,$$ and the resulting prior distribution induced on $(\sigma^2, b)= (\sigma^2, ((\sigma^2)' 
+ \sigma^2 H')/2)$ is denoted by $\Pi=\Pi_{L_n, \bar L_n}$.

\begin{proposition} \label{uniform}
Let $\sigma_0, \mu_0$ satisfy Assumption~\ref{decayassumpt} for some $s\ge 2,B>0$ and choose $$\eps_n=n^{-(s+1)/(2s+3)} (\log n)^\eta, ~~\eta= \frac{s-1}{2s+3}.$$ Let $\Pi=\Pi_{L_n, \bar L}$ be the preceding prior and, if $l_n = \min(L_n, \bar L_n)<\infty$, assume $2^{-l_n(s+1)}\lesssim\eps_n$. Then $\Pi$ satisfies the hypotheses of Theorem~\ref{thm:main} 
for this choice of $\eps_n$, all $D$ large and all $d>0$ small enough. As a consequence 
the resulting posterior distribution contracts about the true parameter 
$(\sigma_0^2, b_0)$ at the minimax optimal rate within $\log n$ factors 
(Remark~\ref{opt}).
\end{proposition}

\begin{remark} \label{coupling} \normalfont \textit{Coupling of smoothness 
indices and ill-posed inverse problems.}
In contrast to estimation of $\sigma^2$ and $b$, estimation of $\mu$ is not ill-posed and possible 
at the standard nonparametric rate $n^{-\alpha/(2\alpha+1)}$, when $\mu$ is 
$\alpha$-smooth. These estimation problems are, however, interacting with each 
other. On the one hand $\sigma, b$ and $\mu$ are closely related by classical 
identities (e.g., (\ref{hierid}) and (\ref{invid})), making it natural that 
$\sigma$ is modelled one degree smoother than $b$. On the other hand the 
smoothness of $\mu$ is \textit{not} identified by the smoothness of $b$ and $\sigma$ 
(for example, even for non-regular $b, \sigma$ the invariant density $\mu$ can be 
very smooth, e.g., constant on $[0,1]$). Our results rely on the 
assumption that $\mu_0$ is at least $s+1$ smooth (whenever $(\sigma, b) \in 
\Theta_s$), \textit{and} that the prior implicitly models the regularity of $\mu$ correctly. This is related to the fact, made explicit in the proofs that 
follow, that the information distance between samples $X_0, X_\Delta, ..., 
X_{n\Delta}$ from parameters $(\sigma, b)$ and $(\sigma_0, b_0)$ \textit{does} 
necessarily involve the $L^2$-distance $\|\mu_{\sigma b}-\mu_{\sigma_0 
b_0}\|_{L^2}$, and the hierarchical prior from above takes this into account.
\end{remark}

\begin{remark} \normalfont \textit{Credible sets.} While our results imply that Bayesian recovery algorithms can be expected 
to work in principle in (scalar) diffusion models, we emphasise that mere 
contraction theorems as those obtained here do not yet justify the use of 
Bayesian posterior inference (`credible regions') in scientific practice. This problem is more 
involved, see the recent paper \cite{SVV15} and its discussion. An interesting 
topic for future research in this direction would be to obtain nonparametric Bernstein-von Mises 
theorems as in \cite{CN13, CN14} for the diffusion model considered here. While 
the contraction results obtained here are useful for this too, obtaining exact 
posterior asymptotics will require a more elaborate analysis.
\end{remark}

\subsubsection{Choice of the prior: extensions and perspectives} \label{outlook} 

In computational practice the methodology closest to the one considered here is described in Section 5.1 of \cite{PapaspiliopoulosPokernRobertsStuart2012}, where a certain Gaussian prior is chosen for the drift function $b$, while the diffusion coefficient is modelled parametrically. A data augmentation method is devised that allows to sample from the posterior distribution (\ref{post}) in this situation. The random wavelet series priors on $b$ from the previous subsection allow for truncated Gaussian priors -- by choosing $B$ large enough our theory can approximate the case of a Gaussian prior on the drift function at least in practice. From a rigorous point of view, however, our proofs rely fundamentally on the technical restriction that the prior for $(\sigma, b)$ concentrates on a fixed smoothness ball in $C^2 \times C^1$, a condition \textit{not} satisfied by Gaussian priors. Whether it can be relaxed is not clear: for instance, the constant $C$ in the Gaussian-type 
tails $e^{-Cn\eps_n^2}$ of our tests scales unfavourably as $C\approx e^{-\|b\|_\infty}$. This is not an artefact of our concentration inequalities but corresponds precisely to the connection between mixing times of Markov chains and their spectral gap (see also  \cite{P15}).

\smallskip

The wavelet priors used in the present paper are convenient in our proofs. They could be replaced by $B$-spline basis priors with random coefficients. In fact, $B$-spline bases generate wavelet bases by a simple Gram-Schmidt ortho-normalisation step (see \cite{M92}, Section 1.3 and p.74), so that proofs would go through with only formal (but notationally cumbersome) changes.

Another extension of interest would be to allow for `adaptive' priors that select the generally unknown smoothness degree $s$ by a hyper-prior. While in principle such results should be within the scope of our techniques, they would require significant modification of the spectral bias estimates from \cite{gobetETAL2004}, and this is left for future research.

\section{Proofs I: Concentration inequalities for reflected 
diffusions}\label{diffdet}

\subsection{Definitions and transition densities}

Let $b:[0,1]\to\R$ be measurable and bounded, let $\sigma:[0,1]\to(0,\infty)$ 
be continuous and let $\nu:[0,1]\to\R$ 
satisfy $\nu(0)=1$, $\nu(1)=-1$. Consider the  reflected diffusion on 
$[0,1]$
\begin{align} \label{diffmref}
 \d X_t=b(X_t)\d t+\sigma(X_t)\d W_t+\nu(X_t)\d L_t(X).
\end{align}
Here $(W_t: t \ge 0)$ is a standard Brownian motion 
and $(L_t(X): t \ge 0)$ is a non-anticipative continuous nondecreasing process 
which 
increases only for $X_t\in\{0,1\}$. This model is considered in 
\cite{gobetETAL2004}, and under the above conditions there exists a weak 
solution $(X_t: t\ge 0)$ of the SDE, see 
\cite{StroockVaradhan1971}. An in our setting equivalent construction of this 
reflected diffusion is by extending $b$ and $\sigma$ to be defined on $\mathbb 
R$ as follows: First we extend the functions to $(-1,1]$ by 
$\sigma(x)=\sigma(-x)$ and $b(x)=-b(-x)$ for $x\in(-1,0)$ and second to 
$\R$ by $\sigma(x)=\sigma(x+2k)$ and $b(x)=b(x+2k)$ for all $x\in\R$ and 
$k\in\Z$ such that $x+2k\in(-1,1]$. If $(\sigma, b) \in \Theta$ then the so 
extended functions $\bar \sigma$ and $\bar b$ are bounded Lipschitz functions 
on 
$\mathbb R$ and we can define the strong Markov process $(Y_t: t \ge 0)$ as the 
pathwise solution of the  equation 
\begin{equation}\label{diffext}
\d Y_t=\bar b(Y_t)\d t+\bar \sigma(Y_t)\d \bar W_t
\end{equation}
on the whole of $\R$ (see Theorem 24.2 and 39.2 in \cite{Bass2011}), where 
$\bar W_t$ is another Brownian motion. A version of the process $(X_t: t \ge 
0)$ 
can then be obtained from $(Y_t: t \ge 0)$ by a simple projection described in 
the proof of the following proposition.

By standard results for one-dimensional diffusions (e.g., \cite{bass1998}, 
Chapter 4), the invariant density of the Markov process $(X_t: t \ge 0)$ is 
given by 
\begin{equation} \label{invid}
 \mu(x)=\mu_{\sigma b}(x)=\frac1{G\sigma^2(x)}\exp\left(\int_{0}^{x}\frac{2b(y)}{\sigma^2(y)}\d y\right),~~x \in [0,1],
\end{equation}
with normalising constant
\begin{equation}\label{normconst}
G:=G_{\sigma b} = 
\int_{0}^{1}\frac1{\sigma^{2}(y)}\exp\left(\int_0^y\frac{2b(z)}{
\sigma^2(z)}\d z\right)\d y.
\end{equation}
We see that whenever $b$ and $\sigma$ are bounded and 
$\sigma$ is bounded away from zero, the invariant density is bounded and 
bounded 
away from zero. Under the stronger assumption $(\sigma,b)\in\Theta$ we can 
obtain a similar result also for the transition densities $p_{\sigma b}(\Delta, 
x,y)$ of the corresponding Markov process $(X_t: t \ge 0)$. The proof is given in the appendix.

\begin{proposition}\label{prop:transition}
 Let $(\sigma,b)\in\Theta$. Then there are constants $0<K'<K<\infty$ depending 
only on $D,d$ such that 
$K' \le p_{\sigma b}(\Delta,x,y)\le K$
for all $x,y\in[0,1]$.
\end{proposition}

\subsection{A Bernstein type inequality}

\begin{theorem}
Let the time difference between observations $\Delta>0$ and constants $D,d>0$
in the definition of $\Theta$ be given. Then there exists $\kappa>0$ depending only
on
$\Delta,D,d>0$ such that for all reflected diffusions (\ref{diffmref}) with
$(\sigma,b)\in\Theta$ and arbitrary initial distribution,
for all bounded
functions $f:[0,1]\to\R$, all $r>0$, all $n \in \mathbb N$, and $Z=\sum_{j=0}^{n-1} (f(X_{j\Delta})-\mathbb E_\mu[f(X_0)])$,
\begin{align*}
&\PP(|Z|>r) \le
\kappa \exp\Big(-\frac{1}{\kappa}\min\Big(\frac{r^2}{n\|f\|^2_{L^2(\mu)}},\frac{ r} { \log (n) \|f\|_\infty} \Big)\Big)
\end{align*}
\end{theorem}
\begin{proof}
We make use of the concentration inequality given in Theorem~6 in
\cite{Adamczak2008} with $m=1$ and verify the assumptions by using results in
\cite{Baxendale2005}.
Let $X_0,X_1,\dots$ be a Markov chain with values in $(S,\mathcal B)$. For
$x\in S$ and $A\in\mathcal B$ we introduce the
transition kernels $P(x,A)=\PP(X_1\in A|X_0=x)$ and 
$P^n(x,A)=\PP(X_n\in A|X_0=x)$. For a
measurable function $V:S\to\R$ we define $P V(x)=\E[V(X_1)|X_0=x]$. The 
following three assumptions are assumed in 
\cite{Baxendale2005}, where we slightly
strengthen the minorization condition to be compatible with the assumption in
\cite{Adamczak2008}.
\begin{enumerate}[({A}1)]
\item\emph{Minorization condition.} There exists $C\in\mathcal B$, $\tilde
\beta>0$ and a probability measure $\nu$ on $(S,\mathcal B)$ such that for all
$x\in C$ and $A\in\mathcal B$
\[P(x,A)\ge\tilde \beta \nu(A),\]
as well as for all $x\in S$ there exists $n\in\N$ such that 
$P^n(x,C)>0$.\label{mino}
\item\emph{Drift condition.} There exist a measurable function
$V:S\to[1,\infty)$
and constants $\lambda<1$ and $K<\infty$ satisfying
\[PV(x)\le\left\{
\begin{array}{ll}
\lambda V(x),& \text{ if }x\notin C,\\
K,& \text{ if }x\in C.
\end{array}
\right.
\]\label{drift}
\item\emph{Strong aperiodicity condition.} There exists $\beta>0$ such that
$\tilde\beta\nu(C)\ge\beta$.\label{aperi}
\end{enumerate}
The conditions (A\ref{mino})-(A\ref{aperi}) are verified for the
reflected diffusion as follows: Let $C=[0,1]$, $\tilde \beta$ the uniform lower 
bound on the
transition density given by Proposition~\ref{prop:transition} and $\nu$ be the
uniform distribution on $[0,1]$. Then
(A\ref{mino}) is satisfied. For (A\ref{drift}) we can take $V$ constant to one
and $K=1$. And (A\ref{aperi}) is satisfied with $\beta=\tilde\beta$.

By Proposition~4.1(ii) and Proposition~4.4, eq.(21) in \cite{Baxendale2005} the
constant $\tau$ in Theorem~6 in \cite{Adamczak2008} is finite.
Under conditions (A\ref{mino})-(A\ref{aperi}) there exists a unique invariant
measure $\mu$ (Theorem 1 in \cite{Baxendale2005}). Using Corollary~6.1 in
\cite{Baxendale2005} and that the Markov chain is reversible we obtain that
for all $f\in L^2(\mu)$
\begin{align}\label{contractionproperty}
\left\|P^n f-\int f\d \mu\right\|_{L^2(\mu)}\le \rho^n\left\|f-\int f\d
\mu\right\|_{L^2(\mu)}
\end{align}
for some $\rho<1$.

Since we are in the case $m=1$, the quantity $(\E [T_2])^{-1}\Var Z_1$ in
\cite{Adamczak2008} is equal to the asymptotic variance, see the third remark
after Theorem 6 there. We bound the asymptotic variance, using (\ref{contractionproperty}) and the Cauchy-Schwarz inequality (see also \eqref{varbound} in the Appendix), by
\begin{align*}
\lim_{n\to\infty}n^{-1}\Var_\mu\left(\sum_{j=0}^{n-1}f(X_{j\Delta})\right)
\le 
\frac{1+\rho}{1-\rho}\Var\left(f(X_0)\right)
\le 
\frac{1+\rho}{1-\rho}\left\|f\right\|_{L^2(\mu)}^2.
\end{align*}
This establishes the concentration inequality for a fixed pair
$(\sigma,b)\in\Theta$.
The constants~$\tau$ and~$\rho$ can be chosen uniformly for the class $\Theta$
since
there is a common lower bound
$\tilde \beta$ on the transition densities.
\end{proof}

Note that the above proof can be generalised to diffusions on $\R$, arguing along the lines of \cite{SoehlTrabs2014}. 

The following generalisation of the previous theorem to bivariate Markov chains $(X_{\Delta j},X_{\Delta(j+1)}: j \in \mathbb N)$ with invariant measure  $\mu_2(x,y) = p(\Delta, x,y) \mu(x)$ is obtained in a similar way, see the appendix for a proof.
\begin{theorem}\label{coninequmulti}
Let $\Delta,D,d>0$ be given. Then there exists $\kappa>0$
depending only on
$\Delta,D,d>0$ such that for all reflected diffusions (\ref{diffmref}) with
$(\sigma,b)\in\Theta$ and arbitrary initial distributions,
for all bounded
functions $f:[0,1]^2\to\R$, all $r>0, n \in \mathbb N$, and $Z=\sum_{j=0}^{n-1}
(f(X_{j\Delta},X_{(j+1)\Delta})-\E_{\mu_2}[f(X_0, X_\Delta)])$,
\begin{align*}
\PP(|Z|>r) \le \kappa\exp\Big(-\frac{1}{\kappa}\min\Big(\frac{r^2}{n\|f\|^2_{L^2(\mu_2)}},
\frac{ r} {\log(n)\|f\|_\infty} \Big)\Big).
\end{align*}
\end{theorem}

\subsection{Concentration inequality for suprema of empirical processes}

Let $\F$ be a class of functions.
For $f\in\F$ let either $Z(f)=\sum_{j=0}^{n-1}(f(X_{j\Delta})-\E_\mu[f(X_0)])$
or 
$Z(f)=\sum_{j=0}^{n-1}(f(X_{j\Delta}, X_{(j+1)\Delta})-\E_{\mu_2}[f(X_0, X_\Delta)])$.
By a change of variables we can rewrite the previous concentration 
inequalities as
\begin{equation}\label{bern}
 \PP\left(|Z(f)|>\max(\sqrt{v^2 x},ux)\right)\le \kappa e^{-x},
\end{equation}
where $u=\kappa\log(n)\|f\|_\infty$ and $v^2=\kappa n \|f\|^2_{L^2(\mu)}$ or 
$v^2=\kappa n \|f\|^2_{L^2(\mu_2)}$.
Let $I$ be a subset of a linear space of finite dimension $d$, and consider a 
class of bounded measurable functions $\F=\{f_i:i\in I\}$ indexed by $I$, and 
such that $0 \in \mathcal F$. We define 
$V^2= \sup_{f \in \mathcal F} v^2,~U= \kappa \log n \sup_{f \in \mathcal 
F}\|f\|_{\infty}$ to obtain the following functional concentration inequality:
\begin{theorem}\label{empconinequ}
For $\tilde\kappa=18$ and for all $x\ge0$ we have
\[\PP\left(\sup_{f\in\F}|Z(f)|\ge\tilde\kappa\left(\sqrt
{V^2(d+x)}
+U(d+x)\right)\right)\le2\kappa e^{-x}.\]
\end{theorem}
Given \eqref{bern}, Theorem \ref{empconinequ} follows by the usual chaining 
argument for empirical processes, given for instance in the form of Theorem~2.1 
in Baraud \cite{Baraud2010}. Baraud's proof applies directly in our setting, 
where we 
notice that his Assumption~2.1 can be replaced by~\eqref{bern}, since that 
assumption is only used to apply Bernstein's inequality in the form 
\eqref{bern}.

\section{Proofs II:  The main contraction Theorem \ref{thm:main}}

Our strategy to prove the main theorem of this article is as follows: In the 
spirit of \cite{GhosalGhoshvanderVaart2000, GhosalvanderVaart2007} we first 
derive a general contraction theorem for discretely sampled 
diffusions that requires 
the prior to charge small neighbourhoods of the true 
parameter measured in the information distance (a version of the 
KL-divergence), 
and that admits the existence of certain frequentist tests uniformly in the 
parameter space. We then show how the information distance can be controlled by 
suitable dual Besov norms, and use the concentration inequalities from the 
previous subsection to construct suitable tests. 

\subsection{General contraction theorem with tests}

We denote by $\K(\PP,\QQ):=\E_{\PP}[\log\frac{\d\PP}{\d\QQ}]$ the 
Kullback--Leibler divergence between two probability measures $\PP$ and $\QQ$ 
defined on the same $\sigma$-algebra.
We write $\PP_{\sigma b}$ and $\E_{\sigma b}$ for the probability and the 
expectation with respect to the reflected diffusion started in the invariant 
distribution.
We also introduce the notation
\[\KL((\sigma_0,b_0),(\sigma,b)):=\E_{\sigma_0 
b_0}\left[\log\left(\frac{p_{\sigma_0 b_0}(\Delta,X_{0} , X_ {\Delta } ) } { 
p_{\sigma b}(\Delta,X_{0},X_{\Delta})}\right)\right],\]
and for every $\eps,\kappa>0$ we define
\begin{align}
  B_{\eps,\kappa}&=\bigg\{\vartheta=(\sigma,b)\in\Theta: 
\KL((\sigma_0,b_0),(\sigma,b))\le\eps^2, \notag\\
&\qquad\left.  \Var_{\sigma_0 b_0}
\left(\log\frac{p_{\sigma b}(\Delta,X_{0},X_{\Delta})}{p_{\sigma_0 b_0}(\Delta, 
X_{0},X_{\Delta})}
\right)\le2\eps^2,\right.\notag\\
&\qquad 
\K(\mu_{\sigma_0 b_0},\mu_{\sigma b})\le \kappa,
\Var_{\sigma_0 b_0} \left(\log\frac{\mu_{\sigma b}(X_0)}{\mu_{\sigma_0 b_0}(X_0)}\right)\le 
2\kappa \bigg\}.\label{bek}
\end{align}

\begin{theorem}\label{thm:contract}
Let $\Pi=\Pi_n$ be a sequence of prior distributions on a $\sigma$-field $\mathcal S$ of subsets of $\Theta$ and 
suppose that $X_0,\dots ,X_{n\Delta}$ are discrete observations of a reflected
diffusion process \eqref{diffmref}, started in the stationary 
distribution $\mu$. Let $\Pi(\cdot|X_0,\dots,X_{n\Delta})$ be the resulting 
posterior distribution (\ref{post}). For $\vartheta_0 = (\sigma_0, b_0) \in \Theta$, $\eps_n$ a 
sequence of positive real numbers such that $\eps_n\to0,$ 
$\sqrt{n}\eps_n\to\infty,$ and $C,\kappa$ fixed positive constants, suppose $\Pi$ 
satisfies for all $n$ large enough
\begin{align}\label{eq:beksmallball}
 \Pi\left(B_{\eps_n,\kappa}\right)\ge e^{-C n\eps_n^2}.
\end{align}
Assume moreover that there exists $0<\bar L<\infty$ such that $\Pi(\Theta\backslash \mathcal B_n)\le \bar Le^{-(C+4) 
n\eps_n^2}$ for some sequence $\mathcal B_n \subset \Theta$ for which we can 
find a sequence of tests (indicator functions) 
$\Psi_n\equiv\Psi(X_0,\dots,X_{n\Delta})$ and of distance functions $d_n$ such that for 
every $n\in \N$, 
$M>0$ large enough,
\[
 \E_{\sigma_0 b_0}[\Psi_n]\to_{n\to\infty}0,\quad\sup_{(\sigma,b)\in\mathcal 
B_n:d_n((\sigma,b),(\sigma_0,b_0))\ge M\eps_n}\E_{\sigma b}[1-\Psi_n]\le \bar L 
e^{-(C+4) n\eps_n^2}.
\]
Then the posterior distribution $\Pi(\cdot|X_0,\dots,X_{n\Delta})$ contracts about 
$(\sigma_0,b_0)$ 
at rate $\eps_n$ in the distance $d_n$, that is, in $\PP_{\sigma_0 b_0}$-probability, as 
$n\to\infty$,
\[
\Pi((\sigma,b):d_n((\sigma,b),(\sigma_0,b_0))>M\eps_n|X_0,\dots,X_{n\Delta}
)\to0.
\]
\end{theorem}

The proof of this 
theorem follows the standard pattern from \cite{GhosalGhoshvanderVaart2000, 
GhosalvanderVaart2007} (see also Section 7.3.1 in \cite{GineNickl2015}), and is given in the appendix.

\subsection{Small ball lemma}

Recall the sets $B_{\eps,\kappa}$ defined in \eqref{bek}. 
\begin{lemma}\label{suplem}
There exists a constant $\bar C>0$ such that for every $\kappa>0$ and for all 
$\eps>0$ small enough
\[\Big\{(\sigma,b)\in\Theta :  \|\mu-\mu_0\|_{L^2([0,1])}
+\|\sigma^{-2}-\sigma_0^{-2}\|_{(B^1_{1\infty})^*}
+\|b-b_0\|_{(B^2_{1\infty})^*}<\frac{\eps}{\bar C}\Big\}\subset 
B_{\eps,\kappa}.\]
\end{lemma}

A crucial step in the proof of this key lemma is the observation, partly borrowed 
from \cite{gobetETAL2004}, that the $L^2([0,1]^2)$-distance between the 
transitions densities $p_{\sigma b}, p_{\sigma_0 b_0}$ is related to a 
suitable Hilbert--Schmidt (HS) norm of the difference between the corresponding 
transition operators. Using the semigroup representation $P_\Delta=e^{\Delta 
L}$ of the transition 
operators $P_{\Delta}$ we can then approximate the information distance on the 
underlying 
experiment by the HS-distance between the corresponding inverse operators of the 
infinitesimal generators $L$ of the underlying diffusions. In turn we can 
obtain 
analytic expressions for the Green's function of the inverses of these 
generators, which ultimately gives the reduction to the dual Besov norms 
appearing above. We split the proof into several steps, given in the following 
subsections.

\subsubsection{The infinitesimal generator $L$ and its inverse} \label{invsec}

We begin by defining the function $S(\cdot)=1/s'(\cdot)$, derived from the 
scale function $s(\cdot)$,
$$S(x):=\frac1{2}\sigma^2(x)\mu(x)=\frac1{2G}\exp\left(\int_{0}^{x}\frac{2b(y)}{
\sigma^2(y)}\d 
y\right),$$
with $G$ the normalising constant of the invariant 
density as in~\eqref{normconst}.
The infinitesimal generator $L=L_{\sigma b}$ of the diffusion (\ref{diffmref}) 
is given by the action
\begin{equation} \label{scal}
 Lf(x)=\frac{1}{2}\sigma^2(x)f''(x)+b(x)f'(x)=\frac1{\mu(x)}(S(x)f'(x))'
\end{equation}
where the domain of this unbounded operator on $L^2(\mu)$ is the 
subspace of the $L^2$-Sobolev space $H^2$ with Neumann boundary conditions
\[
 \dom(L)=\{f\in H^2([0,1]):f'(0)=f'(1)=0\}.
\]
We fix the invariant measure $\mu_0$ belonging to 
$\sigma_0$ and $b_0$ and 
consider $L=L_{\sigma b}$ on $L^2(\mu_0)$, which by the bound from above and 
away from zero 
of the invariant densities is the same set of functions as $L^2(\mu)$.
We introduce
\[
 V:=\left\{f\in L^2(\mu_0):\int_0^1 f \d \mu_0=0\right\}  \text{~~ and ~~} 
V^\perp :=\{f\in L^2(\mu_0): f \text{ constant}\}.
\]
We denote by $\mu_0(L)$ the operator that sends $f$ to the constant function 
$\mu_0(Lf)=\int Lf(x)\mu_0(x)\d x$.
We observe that the operator $L-\mu_0(L)$ leaves the space $V$ invariant, and 
denote by $(L-\mu_0(L))|_V$ its restriction to $V$. Next we introduce 
an 
integral operator $J$ and show that $J$ is an explicit representation of the 
inverse of 
$(L-\mu_0(L))|_V$.
We define
\[
 Jf(x)=\int_0^1 K(x,z)f(z)\mu_0(z)\d z, ~~~ f \in V,
\]
with kernel $K= K_{\sigma b}$ defined as
\[K(x,z)=2G \left(H(x,z)-\frac{\mu(z)}{\mu_0(z)}\int_0^1H(x,y)\mu_0(y)\d 
y\right)\]
where
\begin{align*}
 &H(x,z) = H_{\sigma 
b}(x,z)\\
&=\int_0^1\left(\1_{[z,x]}(y)-\1_{[z,1]}(y)\int_y^1\mu_0(x)\d 
x\right)\exp\left(-\int_0^y\frac{2b(v)}{\sigma^2(v)}\d v\right)\d 
y\frac{\mu(z)}{\mu_0(z)}.
\end{align*}
Here $\1_{[z,x]}=0$ if $x<z$. Writing $\1_{[z,x]}(y) = \1_{[0,x]}(y) \1_{[z,1]}(y)$ and using $\1_{[z,1]}(y)=\1_{[0,y]}(z)$ as well as Fubini's theorem, an alternative representation of $J$ is given by
\begin{align*}
 Jf(x)=2G\int_0^1\int_0^y\left(f(z)-\int_0^1 f \d\mu\right)\mu(z)\d 
z\left(\1_{[0,x]}(y)-\int_y^1\d\mu_0\right)\\
\exp\left(-\int_0^y\frac{2b(v)}{
\sigma^2(v)}\d v\right)\d y.
\end{align*}
We compute the first two derivatives
\begin{align}
&\frac{\d }{\d x}(Jf)(x)=2G\int_0^x\left(f(z)-\int f \d\mu\right)\mu(z)\d z 
\exp\left(-\int_0^x\frac{2b(v)}{\sigma^2(v)}\d v\right),\label{firstderiv}
\\
&\frac{\d^2 }{\d 
x^2}(Jf)(x)=2G\left(f(x)-\int f \d\mu\right)\mu(x)
\exp\left(-\int_0^x\frac{2b(v)}{\sigma^2(v)}\d 
v\right)\notag\\
&+2G\int_0^x\left(f(z)-\int f \d\mu\right)\mu(z)\d z 
\left(-\frac{2b(x)}{\sigma^2(x)}\right)
\exp\left(-\int_0^x\frac{2b(v)}{\sigma^2(v)}\d v\right).\notag
\end{align}
It follows
\begin{align*}
 LJf(x)&=\frac{1}{2}\sigma^2(x)\frac{\d^2 }{\d 
x^2}(Jf)(x)+b(x)\frac{\d }{\d 
x}(Jf)(x)\\
&=\frac{1}{2}\sigma^2(x)2G\left(f(x)-\int f \d\mu\right)\mu(x)
\exp\left(-\int_0^x\frac{2b(v)}{\sigma^2(v)}\d 
v\right)\\
&=f(x)-\int f \d\mu
\end{align*}
and thus 
\(
 (L-\mu_0(L))Jf(x)
=f(x)-\int f \d\mu_0.
\)
Consequently $(L-\mu_0(L))J$ is the identity operator on $V$. We see from the first 
derivative in~\eqref{firstderiv} that 
$Jf\in\dom(L)$. Using $\1_{[0,x]}(y)=\1_{[y,1]}(x)$ and Fubini's theorem one
also sees that $\int_0^1Jf(x)\mu_0(x)\d x=0$ and consequently $Jf\in\dom(L)\cap 
V$.

To see that $(L-\mu_0(L))|_V$ is injective suppose that for 
$f,f_1\in\dom(L)\cap V$ we 
have $(L-\mu_0(L))f=(L-\mu_0(L))f_1$ or equivalently $Lf=Lf_1+c_0$. This 
implies by integration with respect to $d\mu(x)$ that 
$S(x)f'(x)=S(x)f_1'(x)+c_0\int_0^x\d\mu+c_1$ with 
$c_1=c_0=0$ since $f'(0)=f'_1(0)=f'(1)=f_1'(1)=0$. Another integration gives 
$f(x)=f_1(x)+c_2$ 
and $c_2=0$ by $f,f_1\in V$. We conclude that $f=f_1$ showing that 
$(L-\mu_0(L))|_V$ 
is injective, and by what precedes the inverse mapping $(L-\mu_0(L))|_V^{-1}: V 
\to \dom(L) \cap V$ exists, and has integral representation $J$. Note that when $L=L_{\sigma_0 b_0}$ then in view of (\ref{scal}) we have $\mu_0(L_{\sigma_0 b_0})(f)=0$ for all $f \in \dom (L)$ and hence the same integral representation follows for $(L_{\sigma_0 b_0})|_V^{-1}=(L_{\sigma_0 b_0}-\mu_0(L_{\sigma_0 b_0}))|_V^{-1}$.

\smallskip

The following lemma bounds the HS-norm distance between the Green kernels and 
is proved in the appendix.
\begin{lemma}\label{lem:K}
There exists $\tilde C>0$ such that for all $(\sigma,b),(\sigma_0,b_0)\in\Theta$
\begin{align*}
&\quad \left(\int_0^1\int_0^1 (K_{\sigma b}-K_{\sigma_0 
b_0})^2(x,z)\mu_0(x)\mu_0(z)\d x \d 
z\right)^{1/2}\\
&\le \tilde C \|\mu_{\sigma b}-\mu_0\|_{L^2([0,1])}
+\tilde C
\left\|\frac{1}{\sigma^2}-\frac{1}{\sigma_0^2}\right\|_{(B^{1}_{1\infty} )^* }
+\tilde C \left\|b-b_0\right\|_{(B^{2}_{1\infty})^*}.
\end{align*}
\end{lemma}

\subsubsection{Proof of Lemma~\ref{suplem}}
We have
\[
\KL((\sigma_0,b_0),(\sigma,b))
=\E_{\sigma_0 
b_0}
\left[\log\left(\frac{\mu_0(X_0)p_{\sigma_0 b_0}(\Delta,X_0,X_\Delta)}{
\mu_0(X_0)p_{\sigma 
b}(\Delta,X_0,X_\Delta)}\right)\right].
\]
We see that this is the Kullback--Leibler divergence between the probability 
measures 
corresponding to the densities  $\mu_0p_{\sigma_0 b_0}=\mu_0(x)p_{\sigma_0 
b_0}(\Delta, x,y)$ and $\mu_0p_{\sigma b}=\mu_0(x)p_{\sigma b}(\Delta, x,y)$ 
with respect to the 
Lebesgue measure on $[0,1]^2$.
By Lemma~8.2 in \cite{GhosalGhoshvanderVaart2000} we have 
\[\KL((\sigma_0,b_0),(\sigma,b))\le2 h^2(\mu_0p_{\sigma b},\mu_0p_{\sigma_0 
b_0})\left\|\frac{p_{\sigma_0 b_0}}{p_{\sigma b}}\right\|_\infty\]
where $h^2(p,q)=\int (\sqrt p - \sqrt q)^2$ is the usual Hellinger distance between two densities $p,q$. The transition densities are bounded from above and from below in view of 
Proposition~\ref{prop:transition}. Thus the Hellinger distance can be bounded by the $L^2$-norm of the 
difference between the densities
\begin{align*}
 h^2(\mu_0p_{\sigma b},\mu_0p_{\sigma_0 
b_0})\le \left\|\frac{1}{\mu_0p_{\sigma_0 b_0}}\right\|_\infty \|\mu_0p_{\sigma 
b}-\mu_0p_{\sigma_0 b_0}\|_{L^2([0,1]^2)}^2.
\end{align*}
We want to bound the last quantity in terms of the Hilbert--Schmidt norm 
distance  
$\|P^{\sigma b}_\Delta-P^{\sigma_0 
b_0}_\Delta\|_{HS}$ between the transition operators of the respective 
diffusions acting on $L^2(\mu_0)$. We have the integral representation
\begin{align*}
 (P^{\sigma b}_\Delta-P^{\sigma_0 
b_0}_\Delta)f&=\int (p_{\sigma b}(x,y)-p_{\sigma_0 b_0}(x,y))f(y)\d y
\\
&=\int \frac{p_{\sigma b}(x,y)-p_{\sigma_0 b_0}(x,y)}{\mu_0(y)}f(y)\mu_0(y)\d y
\end{align*}
and thus the Hilbert--Schmidt norm is given by
\begin{align*}
 \|P^{\sigma b}_\Delta-P^{\sigma_0 
b_0}_\Delta\|^2_{HS}
=\int\int \left(\frac{p_{\sigma b}(x,y)-p_{\sigma_0 
b_0}(x,y)}{\mu_0(y)}\right)^2\mu_0(x)\mu_0(y)\d x\d y\\
=\int\int \left({p_{\sigma b}(x,y)-p_{\sigma_0 
b_0}(x,y)}\right)^2\frac{\mu_0(x)}{\mu_0(y)}\d x\d y.
\end{align*}
In summary we can bound $\KL((\sigma_0,b_0),(\sigma,b))$ by a constant multiple of
\begin{align*}
 \|\mu_0p_{\sigma b}-\mu_0p_{\sigma_0 b_0}\|_{L^2([0,1]^2)}^2 \le  
\|\mu_0\|_\infty^2\|P^{\sigma b}_\Delta-P^{\sigma_0 b_0}_\Delta\|_{HS}^2.
\end{align*}

Let $(e_k)_{k\in\N}$ be any orthonormal basis of $L^2(\mu_0)$ such that 
$e_0=1$ (for instance we can take the eigen-basis of the operator 
$P_\Delta^{\sigma_0 b_0}$). Then
\begin{align*}
 \|P^{\sigma 
b}_\Delta-P^{\sigma_0 
b_0}_\Delta\|_{HS}&=\sum_{k=0}^{\infty}\|(P_\Delta^{\sigma 
b}-P_\Delta^{\sigma_0 b_0})e_k\|^2_{L^2(\mu_0)} =\sum_{k=1}^{\infty}\|(P_\Delta^{\sigma 
b}-P_\Delta^{\sigma_0 b_0})e_k\|^2_{L^2(\mu_0)}.
\end{align*}
We denote by $\mu_0(P_\Delta^{\sigma b})$ the operator that sends $f$ to 
the constant function $\mu_0(P_\Delta^{\sigma b}f)=\int  P_\Delta^{\sigma 
b}f(x)\mu_0(x)\d x$ and write the Hilbert--Schmidt norm as 
\begin{align*}
 \|P^{\sigma 
b}_\Delta-P^{\sigma_0 
b_0}_\Delta\|_{HS}^2
&=\sum_{k=1}^{\infty}
\|(P_\Delta^{\sigma 
b}-\mu_0(P_\Delta^{\sigma 
b})-P_\Delta^{\sigma_0 b_0}+\mu_0(P_\Delta^{\sigma 
b}))e_k\|^2_{L^2(\mu_0)}\\
&=\sum_{k=1}^{\infty}
\|(P_\Delta^{\sigma 
b}-\mu_0(P_\Delta^{\sigma 
b})-P_\Delta^{\sigma_0 b_0})e_k\|^2_{L^2(\mu_0)}+\sum_{k=1}^{\infty}
\mu_0(P_\Delta^{\sigma 
b}e_k)^2,
\end{align*}
where we have used that all $e_k$ with $k\ge1$ are orthogonal to $e_0$ and thus 
to all constant functions, and that $P_\Delta^{\sigma b}-\mu_0(P_\Delta^{\sigma 
b}), P_\Delta^{\sigma_0 b_0}$ leave the space $V=\{f \in L^2(\mu_0): \int f 
d\mu_0 = 0\}$ invariant (note that $\mu_0$ is the invariant measure for 
$P_\Delta^{\sigma_0 b_0}$). By the last observation we can write
\begin{align}\label{HSdecomp}
 \|P^{\sigma b}_\Delta-P^{\sigma_0 b_0}_\Delta\|_{HS}^2=
\|(P_\Delta^{\sigma b}-\mu_0(P_\Delta^{\sigma b}))|_V-P_\Delta^{\sigma_0 b_0}|_V\|^2_{HS}+\sum_{k=1}^{\infty}
\mu_0(P_\Delta^{\sigma 
b}e_k)^2.
\end{align}
We represent $P_\Delta^{\sigma_0 b_0}|_V = \exp (\Delta L_{\sigma_0 
b_0}|_V)$ and
\begin{equation*}
(P_\Delta^{\sigma b}-\mu_0(P_\Delta^{\sigma 
b}))|_V=\exp(\Delta(L_{\sigma b}-\mu_0(L_{\sigma b}))|_V),
\end{equation*}
possible since, by standard properties of Markov semigroups, the derivatives $$\frac{d}{d\Delta}(P_\Delta^{\sigma 
b}-\mu_0(P_\Delta^{\sigma b}))|_V(f)= \frac{d}{d\Delta}\exp(\Delta(L_{\sigma b}-\mu_0(L_{\sigma 
b}))|_V)(f)$$ coincide for all $f \in V \cap \dom(L)$, and setting $\Delta=0$ gives the identity $id|_{V}$ on 
both sides in the last but one display.

\smallskip

For every $(\sigma, b)$ the operators $L_{\sigma b}$  have a discrete 
non-positive spectrum $\{\lambda_k: k \in \mathbb N\}$ (e.g., p.97-100 in 
\cite{BGL14}), with eigenfunctions $u_k$ orthonormal in $L^2(\mu)$. One checks 
directly that the spectra of $L_{\sigma_0 b_0}|_V, (L_{\sigma b} - 
\mu_0(L_{\sigma b}))|_V$ equal the ones of $L_{\sigma_0 b_0}, L_{\sigma b}$ but 
with eigenvalue $\{0\}$ removed (corresponding to $u_0=1$), and, in the second 
case, for different eigenfunctions $u_k - \mu_0(u_k)$, orthonormal for the 
scalar product $\langle \cdot, \cdot \rangle_{V,\mu} \equiv \langle \cdot - \int 
(\cdot) d \mu, \cdot - \int (\cdot) d\mu \rangle_\mu$ on $V$. Their inverses 
$L_{\sigma_0 b_0}|_V^{-1}$, $(L_{\sigma b}-\mu_0(L_{\sigma b}))|_V^{-1}$ derived 
in Subsection \ref{invsec} are bounded linear operators on $V$ (in view of their 
integral representations), and in view of their spectral representation they are 
also symmetric and thus self-adjoint for the scalar products $\langle \cdot, 
\cdot \rangle_{\mu_0}$ and $ \langle \cdot, \cdot \rangle_{V, \mu}$ on $V$, 
respectively.  

Using the functional calculus and what precedes we can hence represent 
$(P_\Delta^{\sigma b}-\mu_0(P_\Delta^{\sigma b}))|_V$ and $(P_\Delta^{\sigma_0 b_0})|_{V}$ as the composition of the inverses 
$(L_{\sigma b}-\mu_0(L_{\sigma b}))|_V^{-1}$, $L_{\sigma_0 b_0}|_V^{-1}$ and the 
function $f:z\to \exp(\Delta z^{-1})$, which is Lipschitz continuous on 
$(-\infty,0)$, with Lipschitz constant $\Lambda$. Proposition~\ref{prop:kitt} 
in the Appendix states that if $f$ is Lipschitz continuous with Lipschitz constant 
$\Lambda$ on the union of the spectra of two operators $T_1$ and $T_2$, which 
are self-adjoint with respect to different scalar products 
$\langle\cdot,\cdot\rangle_{V, \mu}$ and $\langle\cdot,\cdot\rangle_{\mu_0}$ then we have
\[
 \|f(T_1)-f(T_2)\|_{HS}\le\Lambda\|T_1-T_2\|_{HS},
\]
where the HS-norm is for operators from $(V,\langle\cdot,\cdot\rangle_{\mu_0})$ to $(V,\langle\cdot,\cdot\rangle_{V, \mu})$ (and where, strictly speaking, the operators in the preceding display are composed with suitable identity operators between these spaces). Since 
$\langle\cdot,\cdot\rangle_{V, \mu}$ 
and $\langle\cdot,\cdot\rangle_{\mu_0}$ induce equivalent norms on $V$, up to 
constants the same holds when the Hilbert--Schmidt norm is interpreted for 
operators from $(V,\langle\cdot,\cdot\rangle_{\mu_0})$ to 
$(V,\langle\cdot,\cdot\rangle_{\mu_0})$.
We thus obtain from the results in Subsection \ref{invsec} that
\begin{align*}
&\|(P_\Delta^{\sigma b}-\mu_0(P_\Delta^{\sigma b}))|_V-P_\Delta^{\sigma_0 
b_0}|_V\|^2_{HS} \lesssim \|(L_{\sigma 
b}-\mu_0(L_{\sigma b}))|_V^{-1}-L_{\sigma_0 b_0}|_V^{-1}\|^2_{HS}\\
&= \int_0^1\int_0^1(K_{\sigma b}-K_{\sigma_0 
b_0})^2(x,z)\mu_0(x)\mu_0(z)\d x \d z\equiv A.
\end{align*}

Before we bound $A$ further, let us next consider the second term in \eqref{HSdecomp}. Using 
$\int_0^1\mu_{\sigma b}(x)p_{\sigma 
b}(\Delta,x,y)\d x=\mu_{\sigma b}(y)$ and Parseval's identity we obtain
\begin{align*}
&\quad \sum_{k=1}^{\infty}\mu_0(P_\Delta^{\sigma 
b}e_k)^2= \sum_{k=1}^{\infty}[\mu_0(P_\Delta^{\sigma b}e_k-e_k)]^2\\
&= \sum_{k=1}^{\infty}\left(\int_0^1 \left[\int_0^1\mu_0(x)p_{\sigma 
b}(\Delta,x,y)e_k(y)\d y-\mu_0(x)e_k(x)\right]\d x\right)^2\\
&\le 2\sum_{k=1}^{\infty}\left(\int_0^1
\int_0^1(\mu_0(x)-\mu_{\sigma b}(x))p_{\sigma 
b}(\Delta,x,y)e_k(y)\d y\d x\right)^2\\
&\quad +2\sum_{k=1}^{\infty}\left(\int_0^1
(\mu_{\sigma b}(x)-\mu_0(x))e_k(x)\d x\right)^2 \displaybreak[0] \\
&= 2\sum_{k=1}^\infty\left\langle\frac{\int_0^1(\mu_{\sigma b}-\mu_0)(x)p_{\sigma 
b}(\Delta,x,\cdot)\d x}{\mu_0}, e_k\right\rangle^2_ {\mu_0 }+
2\sum_{k=1}^\infty\left\langle\frac{\mu_{\sigma 
b}-\mu_0}{\mu_0},e_k\right\rangle^2_{ \mu_0 } \\
&=2\int_0^1\left(\frac{\int_0^1(\mu_{\sigma b}-\mu_0)(x)p_{\sigma 
b}(\Delta,x,y)\d x}{\mu_0(y)}\right)^2
\mu_0(y)\d y \\
& ~+2\int_0^1\left(\frac{\mu_{\sigma 
b}(y)-\mu_0(y)}{\mu_0(y)}\right)^2
\mu_0(y)\d y,
\end{align*}
which can be bounded by $B \approx \|\mu_{\sigma b}-\mu_0\|^2_{L^2([0,1])}$,
using that the transition density is bounded from above and $\mu_0$ away from 
zero. Using Lemma~\ref{lem:K} for term $A$ we obtain, for some constant $C>0$ and $\bar C$ large enough, the overall bound $$\KL((\sigma_0, b_0),(\sigma,b))\le C(A+B) \le \eps^2.$$ 

In order to see
\begin{equation}\label{varlog}
  \Var_{\sigma_0 b_0}
\left(\log\frac{p_{\sigma b}(\Delta,X_{0},X_{\Delta})}{p_{\sigma_0 b_0}(\Delta, 
X_{0},X_{\Delta})}
\right)\le2\eps^2
\end{equation}
we bound the variance by the second moment and use
Lemma~8.3 in \cite{GhosalGhoshvanderVaart2000}, which implies that
\[\E_{\sigma_0 b_0}\left[\left|\log
\frac{\mu_0 p_{\sigma b}}{\mu_0 p_{\sigma_0 b_0}}\right|^2\right]\le4 
h^2(\mu_0p_{\sigma 
b},\mu_0p_{\sigma_0 
b_0})\left\|\frac{p_{\sigma_0 b_0}}{p_{\sigma b}}\right\|_\infty.\]
Proceeding as for the Kullback--Leibler divergence above shows~\eqref{varlog}.

It remains to show that
\begin{align}\label{kappabounds}
 \K(\mu_{\sigma_0 b_0},\mu_{\sigma b})\le \kappa\quad\text{and}\quad
\Var_{\sigma_0 b_0}
\left(\log\frac{\mu_{\sigma b}(X_0)}{\mu_{\sigma_0 b_0}(X_0)}\right)\le 
2\kappa.
\end{align}
By Lemmas~8.2 and~8.3 in \cite{GhosalGhoshvanderVaart2000} it suffices to bound $h^2(\mu_{\sigma 
b},\mu_{\sigma_0 
b_0})\left\|\mu_{\sigma_0 b_0}/\mu_{\sigma b}\right\|_\infty.$ Using that $\mu_{\sigma b}$ and $\mu_{\sigma_0 b_0}$ are bounded from above and 
from below and bounding the Hellinger distance by the $L^2$-norm, the Hellinger distance is bounded by a multiple of
\(\|\mu_{\sigma b}-\mu_{\sigma_0 
b_0}\|_{L^2([0,1])}^2
\). So for $\eps>0$ small 
enough we have \eqref{kappabounds}.

\subsection{Construction of tests}

We will now construct the tests needed in Theorem~\ref{thm:contract}. The tests 
are based on the spectral estimators constructed in \cite{gobetETAL2004}, which 
are defined by
\begin{align}
 \hat \sigma^2(x)&:=\frac{2\Delta^{-1}\log(\hat\kappa_1)\int_0^x\hat 
u_1(y)\hat\mu(y)\d y}{\hat u_1'(x)\hat\mu(x)}, \label{sigest}\\
 \hat b(x)&:=\Delta^{-1}\log(\hat\kappa_1)\frac{\hat u_1(x)\hat 
u_1'(x)\hat\mu(x)-\hat u_1''\int_0^x\hat 
u_1(y)\hat\mu(y)\d y}{\hat u_1'(x)^2\hat\mu(x)}, \label{best}
\end{align}
where $\hat \kappa_1, \hat u_1$ are estimates of the second largest eigenvalue 
and associated eigenfunction of the operator $P^{\sigma b}_\Delta$, and 
where $\hat \mu$ is defined in (\ref{muest}) below. Using the concentration 
inequality Theorem \ref{empconinequ} we can prove the following for these 
estimators.
\begin{theorem}\label{thm:conest}
Let $s >1 $ and let $\eps_n$ be such that $n^{-(s+1)/(2s+3)} \lesssim \eps_n 
\lesssim n^{-3/8} (\log n)^{-1/2}$. For all $D>0$ there exists $R>0$ such that 
for $n$ large enough  we have 
uniformly in  $\vartheta = (\sigma, b) \in \Theta_s$
\begin{align*}
 \PP_{\vartheta}\left(\|\hat \sigma^2-\sigma^2\|_{L^2([A,B])}\ge R n \eps_n^3
\text{ or }
\|\hat b-b\|_{L^2([A,B])}\ge R n^2 \eps_n^5\right) \le e^{-D n \eps_n^2}.
\end{align*}
\end{theorem}
We postpone the proof of Theorem~\ref{thm:conest} to the end of the subsection, but record how it can be used to construct tests for the separation metric
\begin{align}\label{dn}
d_n(\theta,\theta_0)=n^{-1}\eps_n^{-2} \|\sigma^2-\sigma_0^2\|_ { L^2( [A, B ] ) 
}+n^{-2} \eps_n^{-4} \|b-b_0\|_{L^2([A,B])}.
\end{align}
Given Theorem \ref{thm:conest} the proof of the following result is elementary (see the Appendix).
\begin{theorem}\label{thm:tests}
For $\theta_0\in\Theta_s$ there exists a 
sequence of tests (indicator functions) 
$\Psi_n\equiv\Psi(X_0,\dots,X_{n\Delta})$ such that for every $n\in \N, C>0$, 
there exists 
$M=M(C)>0$ large enough such that
\[
\E_{\theta_0}[\Psi_n]\to_{n\to\infty}0,\quad\sup_{\theta\in\Theta_s:d_n(\theta,
\theta_0)\ge M\eps_n}\mathbb E_{\theta}[1-\Psi_n]\le e^{-(C+4)n\eps_n^2}.
\]
\end{theorem}

In the remaining part of this subsection we derive concentration inequalities for 
the successive steps in the estimation procedure and at the end of the section 
we prove Theorem~\ref{thm:conest}.
We denote by $\psi_\lambda$ with $\lambda=(l,k),|\lambda|=l$, a compactly 
supported $L^2$-orthonormal wavelet basis of $L^2([0,1])$ as after (\ref{wavnorm}). Let~$V_J$ be the $L^2$-closed linear spaces 
spanned by the wavelets up to level $|\lambda| \le J$. We define $\pi_J$ to be the 
$L^2$-orthogonal projection onto $V_J$ and $\pi_J^\mu$ to be the 
$L^2(\mu)$-orthogonal projection onto $V_J$. We construct estimators as in 
\cite{gobetETAL2004}.
We estimate the action of the transition operator on the wavelet spaces 
$(\bm{P}^J_\Delta)_{\lambda,\lambda'}:=\langle P^{\sigma b}_\Delta\psi_\lambda,
\psi_{\lambda'}\rangle_\mu$ by
\begin{align*}
(\hat{\bm{P}}_\Delta)_{\lambda,\lambda'}
:=\frac1{2n}\sum_{l=1}^n\left(\psi_\lambda(X_{(l-1)\Delta})\psi_{\lambda'}(X_{
l\Delta})+\psi_{\lambda'}(X_{(l-1)\Delta})\psi_{\lambda}(X_{
l\Delta})\right)
\end{align*}
and the $\dim(V_J)\times\dim(V_J)$-dimensional Gram matrix $\bm{G}$ with 
entries 
$\bm{G}_{\lambda,\lambda'}=\langle\psi_\lambda,\psi_{\lambda'}\rangle_\mu$ by
\begin{align*} 
\hat{\bm{G}}_{\lambda,\lambda'}:=\frac1{n}\left(\frac{\psi_\lambda(X_0)\psi_{
\lambda'}(X_0)}{2}+\frac{\psi_{\lambda}(X_{n\Delta})\psi_{\lambda'}(X_{n\Delta}
)}{2}+\sum_{l=1}^ { n-1 } \psi_\lambda(X_{l\Delta})\psi_{\lambda'}(X_{l\Delta}) 
\right).
\end{align*}
Let $u_1$ be the eigenfunction of $P^{\sigma b}_\Delta$ 
corresponding to the second largest eigenvalue $\kappa_1$. Let $u_1^J$ be the 
eigenfunction belonging to the second largest eigenvalue $\kappa_1^J$ of 
the operator $\pi^\mu_J P^{\sigma b}_\Delta$.
\begin{lemma} \label{eigbd}
 $\|u_1^J\|_\infty$ is bounded uniformly in $(\sigma,b)\in\Theta_s$ 
and $J\in \N$.
\end{lemma}
\begin{proof}
By Lemma~6.6 in \cite{gobetETAL2004}, $\|u_1\|_{H^{s+1}}$ is uniformly bounded in $\Theta_s$.  By Corollary 4.6 in 
\cite{gobetETAL2004} this implies that $\|u_1^J\|_{H^1}$ is uniformly bounded, 
and the Sobolev imbedding implies the result.
\end{proof}

Subsequently, $\|\cdot\|_{\ell^2 \to \ell^2}$ denotes the usual norm of an operator on $\ell^2$.
\begin{lemma}
Let $\bm{u}_1^J$ be the vector associated with the normalised eigenfunction 
$u_1^J$ of $\pi_J^\mu P^{\sigma b}_\Delta$ with eigenvalue $\kappa_1^J$.
Let $J=J_n\to\infty$ as $n\to\infty$ be a sequence of integers and let $n$ be 
such that
 $2^J\le c n^{1/2}/\log n$ for some $c>0$. For (\ref{conmuL2}) assume that also 
$2^{-sJ}\le c \sqrt{2^J/n}$. Then for all $D>0$ 
there exists $C>0, \kappa>0$ such that uniformly over~$\Theta_s$
\begin{align}
 \PP\left(\|(\hat{\bm{P}}_\Delta-\bm{P}_\Delta^J)\bm{u}_1^J\|_{\ell^2}
< C\sqrt{\frac{2^J}{n}}\right)&\ge 1-2\kappa e^{-D 2^J},\label{conPu}\\
 \PP\left(\|(\hat{\bm{G}}-\bm{G})\bm{u}_1^J\|_{\ell^2}
<C\sqrt{\frac{2^J}{n}}\right)&\ge 1-2\kappa e^{-D 2^J},\label{conGu}\\
 \PP\left(\|\hat\mu-\mu\|_{L^2}
<C\sqrt{\frac{2^J}{n}}\right)&\ge 1-2\kappa e^{-D 
2^J},\displaybreak[1]\label{conmuL2}\\
 \PP\left(\|\hat{\bm{G}}-\bm{G}\|_{\ell^2\to\ell^2}
<C\frac{2^{2J}}{\sqrt{n}}\right)&\ge 1-\kappa 2^{J+2} e^{-D 2^J},\label{conG}\\
 \PP\left(\|\hat{\bm{P}}_\Delta-\bm{P}_\Delta^J\|_{\ell^2\to\ell^2}
<C\frac{2^{2J}}{\sqrt{n}}\right)&\ge 1-\kappa 2^{J+2} e^{-D 
2^J}.\displaybreak[1]\label{conP}
\end{align}
Moreover, for all $\delta,D>0$ there exists $n_0$ such that we have for all 
$n\ge n_0$,
\begin{equation}
 \PP\left(\|\hat\mu-\mu\|_{L^\infty([0,1])}<\delta\right)\ge 1-2\kappa e^{-D 
2^J}.\label{conmuLinfty}\\
\end{equation}
\end{lemma}
\begin{proof}
We have
\begin{align*}
 \left[(\hat{\bm{P}}_\Delta-\bm{P}_\Delta^J)\bm{u}_1^J\right]_\lambda
&=\frac1{2n}\sum_{l=1}^n\left(\psi_\lambda(X_{(l-1)\Delta})u_1^J(X_{
l\Delta})+u_1^J(X_{(l-1)\Delta})\psi_{\lambda}(X_{
l\Delta})\right.\\
&\qquad \qquad\qquad\left.-\E\left[\psi_\lambda(X_0)
u_1^J(X_\Delta)+u_1^J (X_0)\psi_\lambda(X_\Delta)\right]\right).
\end{align*}
We express the $\ell^2$-norm $\|(\hat{\bm{P}}_\Delta-\bm{P}_\Delta^J)\bm{u}_1^J\|_{\ell^2}$ by its dual representation
\begin{align*}
&\sup_{\|v\|_{\ell^2}\le1}\bigg|\sum_{|\lambda|\le 
J}v_\lambda\frac1{2n}\sum_{l=1 } ^n\left(\psi_\lambda(X_{(l-1)\Delta})u_1^J(X_{
l\Delta})+u_1^J(X_{(l-1)\Delta})\psi_{\lambda}(X_{
l\Delta})\right.\\
&\qquad \qquad\qquad\left.-\E\left[\psi_\lambda(X_0)
u_1^J(X_\Delta)+u_1^J (X_0)\psi_\lambda(X_\Delta)\right]\right)\bigg| \\
&=\sup_{\|v\|_{L^2}\le1, v\in V_J}\bigg|\frac1{2n}\sum_{l=1 } 
^n\left(v(X_{(l-1)\Delta})u_1^J(X_{l\Delta})+u_1^J(X_{(l-1)\Delta})v(X_{
l\Delta})\right. \\
&\qquad \qquad\qquad\left.-\E\left[v(X_0)
u_1^J(X_\Delta)+u_1^J (X_0)v(X_\Delta)\right]\right)\bigg|.
\end{align*}
We consider the class
$\F=\{f(x,y)=(v(x)u_1^J(y)+u_1^J(x)v(y))/2,\|v\|_{L^2}\le1,v\in V_J\}$ with 
$\dim(V_J)=2^{J+1}$. In order to determine $V^2$ and $U$ in Theorem~\ref{empconinequ} we use Lemma 
\ref{eigbd} and calculate 
for $f\in\F$
\begin{align*}
 \|f\|_{L^2(\mu_2)}^2
&=\frac1{4}\|v(x)u_1^J(y)+u_1^J(x)v(y)\|_{L^2(\mu_2)}^2\\
&\le 
\|u_1^J\|_{\infty}^2\|v\|_{L^2(\mu)}^2
\le 
\|u_1^J\|_{\infty}^2\|\mu\|_\infty\|v\|_{L^2}^2
\le C,
\\
\|f\|_\infty
&=\|u^J_1\|_\infty \Big\|\sum_{|\lambda|\le J}\langle v, 
\psi_\lambda\rangle\psi_\lambda \Big\|_\infty
\le C \Big\|\sum_{|\lambda|\le J}|\psi_\lambda| \Big\|_\infty \le C 2^{J/2}.
\end{align*}
We obtain the bounds $V^2 \le \tilde C n$ and
$U\le \tilde C \log(n) 2^{J/2}$ for some constant $\tilde C$. Applying 
Theorem~\ref{empconinequ} yields
\begin{align*}
 \PP\left(\sup_{f\in 
\F}\frac{|Z(f)|}{n}\ge\tilde\kappa\left(\sqrt{\tilde 
C(D+2)\frac{2^{J}}{n}}+\tilde C (D+2)
\frac{\log(n)2^{3J/2}}{n}\right)\right)\le2\kappa e^{-D 2^J}.
\end{align*}
By choice of $J=J_n$ the first term with $\sqrt{2^J/n}$ dominates the second 
term for large $n$ 
and this implies (\ref{conPu}). The bound (\ref{conGu}) for $\hat{\bm{G}}$ 
follows in the same way by considering the class of functions 
$\F=\{f(x,y)=(v(x)u_1^J(x)+v(y)u_1^J(y))/2,\|v\|_{L^2}\le1,v\in V_J\}$.

We denote the empirical measure by 
$\mu_n=\frac1{n+1}\sum_{l=0}^{n}\delta_{X_{l\Delta}}$, and define 
\begin{equation} \label{muest}
\hat \mu=\sum_{|\lambda|\le 
J}\frac{1}{n+1}\sum_{l=0}^{n}\psi_\lambda(X_{l\Delta})\psi_\lambda,
\end{equation}
$K_J(x,y)=\sum_{|\lambda|\le J}\psi_\lambda(x)\psi_\lambda(y)$ and 
$K_J(\mu)=\int K_J(\cdot,y)\mu(y)\d y$.
We consider the variance term $\hat\mu-\pi_J\mu$ and represent, for $B_0$ a 
countable subset of the unit ball $B$ of $L^2([0,1])$,
\[
 \|H\|_{L^2}=\sup_{f\in B_0}\left|\int_{[0,1]}H(t)f(t)\d t\right|,~~~H \in 
L^2([0,1]).
\]
Then $\|\hat\mu-\pi_J \mu\|_{L^2}=\|\mu_n-\mu\|_{\mathcal K}$ with 
$\|H\|_{\mathcal 
K}:=\sup_{k\in\mathcal K}|H(k)|$ and
\[
 \mathcal K:=\left\{
x\mapsto\int_{[0,1]}f(t) K_J(t,x)\d t-\int_{[0,1]}f(t)K_J(\mu)(t)\d t: f\in B_0
\right\}.
\]
We apply the concentration inequality Theorem~\ref{empconinequ} to the class 
$\mathcal K \subset V_J$. As in (20) and (22) in \cite{GineNickl2011} we 
bound 
$\sup_{k\in\mathcal 
K}\|k\|_{\infty}\le C 2^{J/2}$ and $\sup_{k\in\mathcal K}\|k\|^2_{L^2(\mu)}\le 
C$.
We obtain
\begin{align*}
\PP\left(\|\hat\mu-\pi_J\mu\|_{L^2}\ge\tilde\kappa\left(\sqrt
{C (D+2)\frac{2^J}{n}}
+C (D+2)\frac{\log(n) 2^{3J/2}}{n}\right)\right) \le2\kappa e^{-D2^J}.\end{align*}
By choice of $J=J_n$ the term $\sqrt{2^J/n}$ dominates the 
second term and we have for $C>0$ large enough
\[\PP\left(\|\hat\mu-\pi_J\mu\|_{L^2}\ge C\sqrt
{\frac{2^J}{n}}
\right)\le2\kappa e^{-D2^J}.\]
Since $\|\mu\|_{H^s}$ is uniformly bounded over $\Theta_s$ (cf.~(\ref{invid})), 
we have 
$\|\mu-\pi_J \mu\|_{L^2}\le C 2^{-Js}$. By the triangle inequality and the assumption $2^{-Js}\le c\sqrt{2^J/n}$ 
we obtain (\ref{conmuL2}) by possibly increasing the constant $C$. Claim 
(\ref{conmuLinfty}) follows by a similar empirical process type bound for $\mu_n 
- \mu$, corresponding to the case $r=\infty$ in Section 3.1.2 in 
\cite{GineNickl2011}, with $\delta_n = \sqrt n \eps_n^2 \to 0$ there  eventually 
less than any $\delta>0$ (and using Theorem \ref{empconinequ} in place of 
Talagand's inequality). Details are left to the reader.

Next we use the bound 
\begin{equation}\label{HSinequ}
 \|\hat{\bm{G}}-\bm{G}\|_{\ell^2\to\ell^2}\le \sum_{|\lambda|\le J} 
\|(\hat{\bm{G}}-\bm{G})\bm{e}_\lambda\|_{\ell^2},
\end{equation}
where $e_\lambda$ are orthonormal vectors of $(V, \|\cdot\|_{\ell^2})$.
We represent
\begin{align*}
\quad \|(\hat{\bm{G}}-\bm{G})\bm{e}_\lambda\|_{\ell^2} &=\sup_{\|v\|_{L^2}\le1,v\in 
V_J}\bigg|\frac1{n}\sum_{l=1}^n\frac{1}{2}(v(X_{(l-1)\Delta})\psi_\lambda(X_{(l-1)\Delta}) \\
&~~~~~~+v(X_{l\Delta})\psi_\lambda(X_{l\Delta}))-\E[v(X_0)\psi_\lambda(X_0)]\bigg|.
\end{align*}
Similar as before we consider functions of the form
$f(x,y)=(v(x)\psi_\lambda(x)+v(y)\psi_\lambda(y))/2$. Using 
$\|\psi_\lambda\|_\infty\le C 2^{J/2}$ we calculate
$\|f\|^2_{L^2(\mu_2)}\le C 2^J$ and $\|f\|_{\infty}\le C 2^J$. For fixed 
$\lambda$ Theorem \ref{empconinequ} yields the concentration inequality
\begin{align*}
\PP\left( 
\|(\hat{\bm{G}}-\bm{G})\bm{e}_\lambda\|_{L^2}\ge \tilde\kappa\left(\sqrt
{C (D+2)\frac{2^{2J}}{n}}
+C (D+2)\frac{\log(n) 2^{2J}}{n}\right)\right)\le2\kappa e^{-D2^J}.
\end{align*}
The first term in the sum dominates for large $n$. 
Upon choosing a larger constant $C>0$ we obtain
\[\PP\left( 
\|(\hat{\bm{G}}-\bm{G})\bm{e}_\lambda\|_{L^2}\ge C
\frac{2^{J}}{\sqrt{n}}
\right)\le2\kappa e^{-D2^J}.\]
By observing that the sum in~\eqref{HSinequ} is over $2^{J+1}$ summands we 
obtain (\ref{conG}), by enlarging the constant $C$ if necessary.

The final bound (\ref{conP}) for 
$\|\hat{\bm{P}}_\Delta-\bm{P}_\Delta^J\|_{\ell^2\to\ell^2}$ 
follows similarly by considering the functions 
$f(x,y)=(v(x)\psi_\lambda(y)+v(y)\psi_\lambda(x))/2$.
\end{proof}

In the following we assume $2^J \le c n^{1/4}/\log n$ and will say that an event $A$ occurs \emph{with sufficiently high probability} if for 
all $D>0$ there exists $C>0$ such that $\mathbb P(A)$ is at least as large as the probability of the intersection of 
the events in the previous theorem. Then for $n$ large enough the events in 
\eqref{conG}-\eqref{conmuLinfty} include the events that 
$\|\hat{\bm{G}}-\bm{G}\|_{\ell^2 \to \ell^2}\le\tfrac1{2}\|\bm{G}^{-1}\|_{\ell^2 \to \ell^2}^{-1}$, 
$\|\hat\mu-\mu\|_{L^\infty([0,1])}
\le\tfrac{1}{2}\inf_{x \in [0,1]} \mu(x)$. This implies that 
$\hat{\bm{G}}$ is invertible with $\|\hat{\bm{G}}^{-1}\|_{\ell^2 \to \ell^2}\le 
2\|{\bm{G}}^{-1}\|_{\ell^2 \to \ell^2}$ and that $\hat \mu$ is bounded away from zero on $[0,1]$.

\begin{lemma}\label{lem:GPu}
Assume $2^J\le c n^{1/4}/\log n$. For $n$ large enough we have with sufficiently high probability and 
uniformly 
over~$\Theta_s$   \[\|(\hat{{\bm{G}}}^{-1}\hat{\bm{P}}_{\Delta}- 
{\bm{G}}^{-1}{\bm{P}}_{\Delta}^J)\bm{u}_1^J\|_{\ell^2}< C 
\sqrt{\frac{2^J}{n}}.\]
\end{lemma}
\begin{proof}
 We decompose
\begin{align}
  \hat{{\bm{G}}}^{-1}\hat{\bm{P}}_{\Delta}- {\bm{G}}^{-1}{\bm{P}}_{\Delta}^J
&=\hat {\bm{G}}^{-1}(\hat{\bm{P}}_\Delta-{\bm{P}}_\Delta^J)+(\hat 
{\bm{G}}^{-1}- 
{\bm{G}}^{-1}){\bm{P}}_\Delta^J\notag\\
&=\hat {\bm{G}} ^{-1} 
((\hat{\bm{{\bm{P}}}}_\Delta-{\bm{P}}_\Delta^J)+({\bm{G}}- 
\hat {\bm{G}}){\bm{G}}^{-1}{\bm{P}}_\Delta^J).\label{GPdecomp}
\end{align}
Using that $\bm{G}^{-1}\bm{P}^J_{\Delta} \bm{u}_1^J=\kappa_1^J\bm{u}_1^J$ and 
$\|\hat{\bm{G}}^{-1}\|\le 2\|\bm{G}^{-1}\|$ we obtain
\begin{align*}
&   \|(\hat{{\bm{G}}}^{-1}\hat{\bm{P}}_{\Delta}- 
{\bm{G}}^{-1}{\bm{P}}_{\Delta}^J)\bm{u}_1^J\|_{\ell^2}\\
&\le 2\| {\bm{G}} ^{-1}\| 
\left(\|(\hat{\bm{{\bm{P}}}}_\Delta-{\bm{P}}_\Delta^J)\bm{u}_1^J\|_{\ell^2}+\|({
\bm { G } } - 
\hat {\bm{G}})\kappa_1^J\bm{u}_1^J\|_{\ell^2}\right).
\end{align*}
The results follows from this and \eqref{conPu}, \eqref{conGu}.
\end{proof}

\begin{lemma}\label{lem:GP}
Assume $2^J\le c n^{1/4}/\log n$. For $n$ large enough we have with sufficiently high probability and uniformly 
over~$\Theta_s$   
\[\|\hat{{\bm{G}}}^{-1}\hat{\bm{P}}_{\Delta}- 
{\bm{G}}^{-1}{\bm{P}}_{\Delta}^J\|_{\ell^2\to\ell^2}< C 
\frac{2^{2J}}{\sqrt{n}}.\]
\end{lemma}
\begin{proof}
From \eqref{GPdecomp} we deduce
\begin{align*}
  \|\hat{{\bm{G}}}^{-1}\hat{\bm{P}}_{\Delta}- {\bm{G}}^{-1}{\bm{P}}_{\Delta}^J\|
&\le2 \|{\bm{G}} ^{-1}\| 
(\|\hat{\bm{{\bm{P}}}}_\Delta-{\bm{P}}_\Delta^J\|+\|{\bm{G}}- 
\hat {\bm{G}}\|  \|{\bm{G}}^{-1}\|  \|{\bm{P}}_\Delta^J\|)\\
&\le C (\|\hat{\bm{{\bm{P}}}}_\Delta-{\bm{P}}_\Delta^J\|+\|{\bm{G}}- \hat {\bm{G}}\|).
\end{align*}
The result follows by applying the concentration from~\eqref{conG} 
and~\eqref{conP}.
\end{proof}

\begin{lemma}\label{lem:con}
Assume $2^J\le c n^{1/4}/\log n$. Let $\hat \kappa_1$ be the second largest eigenvalue of the matrix $\hat {\bm G} ^{-1} \hat {\bm P}_\Delta$ with corresponding eigenvector $\hat{\bm{u}}_1$ and eigenfunction $\hat u_1 = \sum_\lambda (\hat{\bm{u}}_{1})_{\lambda} \psi_\lambda \in V_J$. For $n$ large enough we have with sufficiently high probability and uniformly 
over~$\Theta_s$   
\begin{align*}
& |\hat\kappa_1-\kappa_1^J|+\|\hat{\bm{u}}_1-{\bm{u}}_1^J\|_{\ell^2}< C \sqrt{\frac{2^{J}}{n}},\\
& \|\hat{{u}}_1-{{u}}_1^J\|_{H^1}< C \sqrt{\frac{2^{3J}}{n}}, ~~~~\|\hat{{u}}_1-{{u}}_1^J\|_{H^2}< C \sqrt{\frac{2^{5J}}{n}}.
\end{align*}
\end{lemma}
\begin{proof}
By Lemma~\ref{lem:GP} we have that $\|\hat{{\bm{G}}}^{-1}\hat{\bm{P}}_{\Delta}- 
{\bm{G}}^{-1}{\bm{P}}_{\Delta}^J\|_{\ell^2\to\ell^2}$ converges to zero. Thus 
the concentration in Lemma~\ref{lem:GPu} carries over to concentration of 
$\hat\kappa_1$ and $\hat{\bm{u}}_1$ by Proposition~4.2 and Corollary~4.3 in 
\cite{gobetETAL2004}. The uniform choice of $\rho$ and $R$ is possible as in 
the 
proof of their Corollary~4.15. The second and third claim are consequences of 
the first by the usual Bernstein inequalities for functions in $V_J$: $\|\hat 
u_1-u_1^J\|_{H^1}\le C 2^J \|\hat 
u_1-u_1^J\|_{L^2}$ and $\|\hat u_1-u_1^J\|_{H^2}\le C 2^{2J} \|\hat 
u_1-u_1^J\|_{L^2}$ (arguing, e.g., as in Proposition 4.2.8 in \cite{GineNickl2015}).
\end{proof}

\begin{proof}[Proof of Theorem \ref{thm:conest}.]
By starting with a slightly larger constant $\tilde D>D$ the factor in front of 
the exponential function can be removed and events of sufficiently high 
probability are seen 
to have probability at least $1-e^{-D2^J}$. We then choose $2^J = n \eps_n^2$. 
By Lemma~\ref{lem:con}, the corresponding bias estimates in 
\cite{gobetETAL2004} and the Sobolev imbedding, we have with sufficiently high 
probability
\begin{align*}
&  |\hat\kappa_1-\kappa_1|+\|\hat u_1-u_1\|_{H^1}< C n \eps_n^3,\\
&  \|\hat u_1'-u_1'\|_{L^\infty([A,B])} \le  \|\hat u_1-u_1\|_{H^2}< C n^2 
\eps_n^5,
\end{align*}
where we used that the bias term is dominated by the variance term 
since $\eps_n \gtrsim n^{-(s+1)/(2s+3)}$. For the estimation of $\hat \mu$ we 
choose $\bar J\ge J$ differently such that 
$2^{\bar J}\sim n^{1/(2s+1)}$ and obtain
\begin{align*}
 \|\hat \mu-\mu\|_{L^2}&< C n^{-s/(2s+1)} = o(n\eps_n^3).
\end{align*}
In addition the event can be chosen such that $\hat \mu$ and $\hat u_1'$ are 
bounded from below on $[A,B]$ uniformly over $\Theta_s$, since $\mu$ and $u_1'$ are (Proposition 6.5 in \cite{gobetETAL2004}). By Lemma~6.6 in 
\cite{gobetETAL2004} we have that $\|u_1\|_{H^{s+1}}$, $s \ge2$, is bounded 
uniformly over $\Theta_s$. This implies in particular uniform bounds for 
$\|u_1\|_{L^2}$, $\|u_1'\|_{L^2}$, $\|u_1''\|_{L^2}$, $\|u_1\|_{\infty}$ and 
$\|u_1'\|_{\infty}$. By the convergence of $\hat u_1$ in $H^2$ these bounds 
carry over to bounds on $\hat u_1$. From the expressions (\ref{sigest}) for 
$\hat \sigma$ and (\ref{best}) for $\hat b$ and the above bounds we deduce 
Theorem \ref{thm:conest}.
\end{proof}

\subsection{Conclusion of the proof of Theorem~\ref{thm:main}}

Theorem~\ref{thm:main} follows from Theorem~\ref{thm:contract}: We choose 
$\mathcal B_n=\Theta_s$ for all $n$. By Lemma~\ref{suplem} there exists $\bar C$ such that
\[\{(\sigma,b)\in\Theta :  \|\mu-\mu_0\|_{L^2([0,1])}+
\|\sigma^{-2}-\sigma_0^{-2}\|_{(B^1_{1\infty})^*}+\|b-b_0\|_{(B^2_{1\infty})^*}<\frac{\eps}{\bar C}\}\subset 
B_{\eps,\kappa}.\]
By assumption we have \eqref{eq:smallball} and by dividing $C$ by $\bar 
C^2$ we ensure \eqref{eq:beksmallball} with a possibly different constant~$C$. 
We define 
$d_n$ as in \eqref{dn}, so that the existence of tests is guaranteed by 
Theorem~\ref{thm:tests}. The result follows.

\section{Proofs III: Wavelet series priors}

We record the following technical lemma whose proof is given in the appendix. Define the dual norm
\begin{equation} \label{dualnorm}
\|f\|_{(B^{s}_{1\infty})^*}:=\sup_{g:\|g\|_{B^s_{1\infty}}\le1}\left|\int_{0
}^1
f(x)g(x)\d x \right|, ~~s \ge 0,
\end{equation}
where the norm of $B^s_{1\infty}$ is defined as in (4.79) or equivalently (4.149) in \cite{GineNickl2015}.
\begin{lemma}\label{besov}
a) Let $f,g$ have $B^1_{\infty \infty}$-norm at most $B'$. Then there exists a 
constant $c(B')$ such that $$\|e^f-e^g\|_{(B^1_{1\infty})^*} \le c(B') 
\|f-g\|_{(B^1_{1\infty})^*}.$$

b) For all $f \in L^\infty$ we have, for $s>0$, 
$$\|f\|_{(B^s_{1\infty})^*}\le\|f\|_{B^{-s}_{\infty 1}} \equiv \sum_{l} 
2^{-l(s-1/2)} \max_k |\langle f, \psi_{lk} \rangle_{L^2([0,1])}|.$$

c) For all $(\sigma, b), (\sigma_0, b_0) \in \Theta$ with corresponding 
invariant measures $\mu, \mu_0$, assuming also that $\sigma, \sigma_0, \mu, \mu_0$ are all periodic on $[0,1]$, we have $$\|b-b_0\|_{(B^2_{1\infty})^*} \lesssim 
\|\mu-\mu_0\|_{L^2} + \|\sigma^{-2} - \sigma_0^{-2}\|_{(B^1_{1\infty})^*}.$$
\end{lemma}

\begin{proof} [Proof of Proposition \ref{uniform}]
We first show that $\Pi(\Theta_s)=1$: By construction of the priors $(\log 
(\sigma^{-2}), \log \mu)$ is almost surely norm-bounded in $\mathcal C^s \times 
\mathcal C^{s+1}$ by $\tilde B$, and this bound carries over to $(\sigma^2, \mu)$ up to 
constants. By (\ref{hierid}) we thus have $\|b\|_{\mathcal C^{s-1}} \lesssim 
\|\sigma^2\|_{\mathcal C^s} + \|\mu\|_{\mathcal C^{s}} \lesssim \tilde B$. Then by 
(\ref{wavnorm}) and the remarks before it we have the continuous imbeddings 
$(\sigma, b) \in (\mathcal C^s \times \mathcal C^{s-1}) \subset (H^s \times H^{s-1}) \cap (B^s_{\infty 1} 
\times B^{s-1}_{\infty 1}) \subset C^2 \times C^1$. Summarising, given $\tilde B$, $(\sigma, b) \in 
\Theta_s$ is true $\Pi$ almost surely for suitable $D=D(\tilde B)$ and $d=d(\tilde B)$. 

To verify the small ball estimate, note that by Lemma \ref{besov}c)  and 
independence of the priors,
\begin{align*}
 &\Pi \left(\theta=(\sigma,b)\in\Theta : \|\mu-\mu_0\|_{L^2} +
\|\sigma^{-2}-\sigma_0^{-2}\|_{(B^1_{1\infty})^*}+\|b-b_0\|_{(B^2_{1\infty})^*}
<\eps_n\right)\\
&\ge 
\Pi\left(\|\sigma^{-2}-\sigma_0^{-2}\|_{(B^1_{1\infty})^*} + 
\|\mu-\mu_0\|_{L^2}<\frac{2\eps_n}{c}\right) \\
&\ge \Pi \left(\|\sigma^{-2}-\sigma_0^{-2}\|_{(B^1_{1\infty})^*} 
<\frac{\eps_n}{c}\right)\PP\left(\|\mu-\mu_0\|_{L^2}<\frac{\eps_n}{c}\right)
\end{align*}
for some constant $c>0$. Examining the first factor we can use 
Lemma~\ref{besov}a),b) and the definition of the Besov norm to obtain the lower bound
\begin{align*}
 &\Pi \left(\|\log \sigma^{-2}- \log \sigma_0^{-2}\|_{B^{-1}_{\infty 
1}}<\frac{\eps_n}{c'(\tilde B)}\right) \\
&=
\PP\left(\sum_{l} 
2^{-l/2}\max_{k}|\tau_{lk}-2^{-l(s+1/2)}l^{-2}u_{lk}|
<\frac{\eps_n}{c'(\tilde B)}\right),
\end{align*}
where $u_{lk}=0$ for all $l>L_n$ (when $L_n<\infty$). We define $t_{lk}=2^{l(s+1/2)}l^2\tau_{lk}$ such that $|t_{lk}| \le \tilde B$,
and $M(J)=\sum_{l=J_0}^{J}\sum_{k=0}^{2^l-1}1\le  2\cdot 2^J$.
We choose $J=J_n$ of order $\eps_n\sim 
2^{-J(s+1)}/J^2$ but such that $\tilde c\eps_n\ge 
2^{-J(s+1)}/J^2$ for some constant $\tilde c>0$ to be determined later. By choice of $L=L_n$ we have $2^{-L(s+1)}\lesssim 2^{-J(s+1)}/J^2$ so that $L$ is eventually larger than $J$.
By choosing $\tilde c>0$ small enough the last probability is bounded below by (all indices 
$(l,k)$ are tacitly assumed to lie in $\mathcal I$ only)
\begin{align*}
&\PP\left(\sum_{l \le J}2^{-l(s+1)}l^{-2}\max_k|t_{lk}-u_{lk}|<\frac{\eps_n}
{c'(\tilde B)}- \bar c 2^{-J(s+1)}/J^2
\right)\\
&\ge
\PP\left(\max_{l\le J}\max_k|t_{lk}-u_{lk}|< c'\eps_n
\right) =
\prod_{l\le J}\prod_k\PP\left(|t_{lk}-u_{lk}|<c'\eps_n
\right)\\
&\ge
\left(\zeta c'\eps_n\right)^{M(J)} \ge e^{-c''(\log 
n)^{1-2/(s+1)}/\eps_n^{1/(s+1)}} \ge e^{-Cn\eps_n^2/2}
\end{align*}
for some constant $C>0$, completing the treatment of this term. For the second 
term notice that since $H, H_0 = \log \mu_0$ are bounded functions the 
exponential map is Lipschitz on the union of their ranges, and thus $\|\mu- 
\mu_0\|_2 \lesssim \|H-H_0\|_\infty$. Then one proves, using $\|h\|_\infty 
\lesssim \sum_{l}  2^{l/2} \max_k|\langle h, \psi_{lk}\rangle|$ and proceeding 
just as above with $\bar u_{lk}=0$ for $l> \bar L_n$, that (again all indices are tacitly assumed to lie in $\mathcal 
I$ only)
\begin{align*}
\PP\left(\|H-H_0\|_\infty<c' \eps_n\right) &\ge \PP\left(\sum_{l} 
2^{l/2}\max_{k}|\beta_{lk}-2^{-l(s+3/2)}l^{-2}\bar u_{lk}|<c''\eps_n \right) 
\end{align*}
is lower bounded by $e^{-Cn\eps_n^2/2}$. We conclude overall that for $n$ large enough,
\begin{align*}
 \Pi\left((\sigma,b)\in\Theta :  
\|\sigma^{-2}-\sigma_0^{-2}\|_{(B^1_{1\infty})^*} 
+\|\mu-\mu_0\|_{L^2}<\eps_n\right)
&\ge e^{-C n\eps_n^2}.
\end{align*}
\end{proof}

\section{Appendix}

\subsection{Proof of Proposition \ref{prop:transition}}
We suppress the subindices $\sigma, b$ in what follows. Define
\begin{align*}
 p:\R\to[0,1],x\mapsto |x+2k|,\qquad \text{with }k\in\Z\text{ such that 
}x+2k\in(-1,1].
\end{align*}Then $(p(Y_t): t \ge 0)$, with $Y_t$ as in (\ref{diffext}), is a 
Markov process whose distribution coincides with the one of $(X_t: t \ge 0)$, 
see I.\S~23 in \cite{gihmanSkorohod1972}. 

Let $p^Y(\Delta,x,y)$ be the transition density of $Y$ from $x$ to $y$. 
The transition density $p^X(\Delta,x,y)$, $x,y\in[0,1]$, of $X$ can be 
written as
\begin{equation} \label{formel}
 p^X(\Delta,x,y)=\sum_{z:p(z)=y} p^Y(\Delta,x,z).
\end{equation}
Since $(\sigma, b) \in \Theta$ the extension $\bar b$ of $b$ is bounded and 
differentiable on $\R$ with bounded derivative and the extension $\bar \sigma$ 
of $\sigma$ is bounded and twice 
differentiable on $\R$ with bounded derivatives. We further recall that $\sigma, 
\bar \sigma$ are bounded away from zero. Define the function 
\[f(x)=\int_0^x\frac{\d y}{\bar \sigma(y)},
\]
and denote by $g$ its inverse function. Thus $f(g(x))=x$, $f'(x)=\frac{1}{\bar 
\sigma(x)}$ and 
$g'(x)=\frac{1}{f'(g(x))}=\bar \sigma(g(x))$. We further define 
\[\bar a(x)=\frac{\bar b(g(x))}{\bar \sigma(g(x))}-\frac{1}{2}\bar 
\sigma'(g(x)),\quad 
\bar B(x)=-\frac{1}{2}\bar a^2(x)-\frac{1}{2}\bar a'(x).\]
The formula for the 
transition density of $(Y_t: t \ge 0)$ is given by~(9) in I.\S~13 of 
\cite{gihmanSkorohod1972} (under hypotheses to be verified) and 
reads, for $\eta^*$ a standard Brownian bridge process,
\begin{align*}
& p^Y(\Delta, x,y)\\
&=\frac{1}{\sqrt{2\pi\Delta}\bar \sigma(y)}
\left(\frac{\bar \sigma(x)}{\bar \sigma(y)} 
\right)^{\frac{1}{2}}\exp\left(-\frac{1}{2\Delta}
\left(\int_x^y\frac{\d z}{\bar \sigma(z)}\right)^2
+\int_x^y\frac{\bar b(z)}{\bar \sigma^2(z)}\d z
\right)\\
&\quad\times \E\left[
\exp\left(
\Delta\int_0^1\bar B(f(x)+\sqrt{\Delta}\eta^*(u)+u[f(y)-f(x)])
\d u\right)
\right].
\end{align*}
Since $\bar B$ involves the derivative $\bar a'$ we can calculate
\begin{align*}
 \bar a'(x)&=\left(
\frac{\bar b'(g(x))\bar \sigma(g(x))-\bar b(g(x))\bar \sigma'(g(x))}{\bar 
\sigma^2(g(x))}-\frac{1}{2}
\bar \sigma''(g(x))\right)g'(x)\\
&=\left(
\frac{\bar b'(g(x))\bar \sigma(g(x))-\bar b(g(x))\bar \sigma'(g(x))}{\bar 
\sigma^2(g(x))}-\frac{1}{2}
\bar \sigma''(g(x))
\right)\bar \sigma(g(x)).
\end{align*}
Especially, we see that $\bar a$ and $\bar a'$ are bounded and consequently 
also $\bar B(x)$ is bounded, in particular 
\[
 \limsup_{|x|\to\infty}\frac{1}{1+x^2}\bar B(x) = 0
\]
so that the above formula for the transition density to holds. We 
conclude from what precedes that uniformly in $x,y\in[0,1]$ the transition 
density $p^Y(\Delta,x,y)$ is bounded away from zero and by (\ref{formel}) the 
same lower bound carries over to $p^X(\Delta, x,y)$. It remains to derive the 
upper bound for $p^X(\Delta, x, y)$, as follows: one shows
\begin{equation}\label{finseries}
\sup_{x,y\in[0,1]}\sum_{z:p(z)=y}\exp\left(-\frac1{2\Delta}\left(\int_x^z\frac{
\d t}{\bar \sigma(t)}\right)^2\right)<\infty
\end{equation}
by the upper bound on $\bar \sigma$. In addition, since for all 
$k\in\Z, x \in \mathbb R$ we have $\int_x^{x+2k}\bar b(t)/\bar \sigma^2(t)\d t=0$ by the construction of the `reflected extension' of $b, \sigma$ described at the 
beginning of this subsection, we have
\begin{align*}
\sup_{x,y\in\R}\exp\left(\int_x^y\frac{\bar b(t)}{\bar \sigma^2(t)}\d 
t\right)\le\sup_{x,y\in[-1,1]}\exp\left(\int_x^y\frac{\bar b(t)}{\bar 
\sigma^2(t)}\d 
t\right),
\end{align*}
and the last expression is uniformly bounded since $\bar b/\bar \sigma^2$ is. 
Combining the above we conclude
$\sup_{x,y\in[0,1]}p^X(\Delta,x,y) \le K$ for some finite constant $K$.

\subsection{Proof of Theorem \ref{thm:contract}}

We start with some preparatory remarks and two lemmas. The transition operator $P$ of a Markov chain $(Y_j)_{j\in\N}$ is defined by 
$Pf(x)=\E[f(Y_1)|Y_0=x]$. We will need to bound 
variances of sums of the form $\sum_{j=1}^{n}f(X_{j\Delta})$ and 
$\sum_{j=1}^{n}f(X_{(j-1)\Delta},X_{j\Delta})$. 

\smallskip

For the first type of sum we can use the contraction 
property~\eqref{contractionproperty}, which states that the transition operator 
$P$ of the Markov chain $(X_{j\Delta})_{j\in\N}$ satisfies 
$\|Pf\|_{L^2(\mu)}\le\rho\|f\|_{L^2(\mu)}$ for some 
$\rho\in(0,1)$ and for all $f\in L^2(\mu)$ with $\int f \d \mu=0$ (where $\mu$ 
is the associated invariant measure). 
As a consequence if for $f\in L^2(\mu)$ we denote $\bar f=f-\int f\d \mu$ then
we have for the diffusions started in the invariant distribution, and using the 
Cauchy-Schwarz inequality,
\begin{align}
\Var\left(\sum_{j=1}^{n}f(X_{j\Delta})\right)&=\sum_{j=1}^{n}\Var\left(f(X_{
j\Delta})\right)+2\sum_{j<k}\text{Cov}\left(f(X_{j\Delta})f(X_{k\Delta}
)\right)\notag\\
&\le n \E[\bar f(X_0)^2]+2n\sum_{\ell=1}^{n-1}\E[\bar f(X_0)\bar 
f(X_{\ell\Delta})]\notag \displaybreak[0] \\
&\le n \|\bar f\|^2_{L^2(\mu)}+2n\sum_{\ell=1}^{n-1}\|\bar 
f\|_{L^2(\mu)}\|P^\ell \bar f\|_{L^2(\mu)}\notag \displaybreak[0]\\
&\le n\left(1+2\sum_{\ell=1}^{n-1}\rho^\ell\right)\Var(f(X_0))
\le n\frac{1+\rho}{1-\rho}\Var(f(X_0)).\label{varbound}
\end{align}
For the second type of sums, of the form 
$\sum_{j=1}^{n}f(X_{(j-1)\Delta},X_{j\Delta})$, the 
contraction property is needed for Markov chains 
$(Y_{2j\Delta},Y_{(2j+1)\Delta})_{j\in\N}$ started in the invariant 
distribution 
$\mu_2 (x,y) = \mu(x) p(\Delta, x,y)$, and the next lemma extends the 
contraction property to such Markov chains.
\begin{lemma}\label{contractlem}
Suppose $(Y_j)_{j\in\N}$ is a Markov chain with invariant density $\mu$, 
transition operator $P$ and 
that 
$P$ is an $L^2(\mu)$-contraction in the sense that for some $\rho\in(0,1)$ we 
have $\|Pg\|_{L^2(\mu)}\le\rho\|g\|_{L^2(\mu)}$ for all $g\in L^2(\mu)$ with 
$\int g\d\mu=0$. Consider the Markov chain 
$(Y_{2j},Y_{2j+1})_{j\in\N}$ with transition operator $P_2$ 
and invariant distribution $\mu_2$. Then~$P_2$ is a $L^2(\mu_2)$-contraction, 
more 
precisely, we have
\[
\|P_2 f\|_{L^2(\mu_2)}\le\rho \|f\|_{L^2(\mu_2)}
\]
for all $f\in L^2(\mu_2)$ with $\int f \d\mu_2=0$.
\end{lemma}
\begin{proof}
We define $g(x):=\E [f(Y_{2j},Y_{2j+1})|Y_{2j}=x]$. By 
the assumption we have
\[
(\E[(\E[g(Y_{2(j+1)})|Y_{2j+1}])^2])^{1/2}\le\rho(\E[g(Y_{2(j+1)}
)^2])^{1/2},
\]
where the expectations are with respect to the stationary distribution.
The left hand side equals $\|P_2 f\|_{L^2(\mu_2)}$.
By Jensen's inequality for conditional expectations we have 
$(\E[g(Y_{2(j+1)})^2])^{1/2}\le\|f\|_{L^2(\mu_2)}$ concluding the proof of 
the lemma.
\end{proof}

\begin{lemma}\label{anlem}
 For every $\eps,\kappa>0$ and probability measure $\nu$ on the set
\begin{align*}
  B_{\eps,\kappa}&=\bigg\{(\sigma,b)\in\Theta: 
\KL((\sigma_0,b_0),(\sigma,b))\le\eps^2, \\
&\qquad\left.
\Var_{\sigma_0 b_0}
\left(\log\frac{p_{\sigma b}(\Delta,X_{0},X_{\Delta})}{p_{\sigma_0 b_0}(\Delta, 
X_{0},X_{\Delta})}
\right)\le2\eps^2,
\right.\\
&\qquad
\K(\mu_{\sigma_0 b_0},\mu_{\sigma b})\le \kappa,
\Var_{\sigma_0 b_0}
\left(\log\frac{\mu_{\sigma b}(X_0)}{\mu_{\sigma_0 b_0}(X_0)}\right)\le 2\kappa \bigg\}
\end{align*}
we have, for every $c>0$,
\begin{align*}
  \PP_{\sigma_0 b_0}\left(\int_{B_{\eps,\kappa}} 
\frac{\mu_{\sigma b}(X_0)}{\mu_{\sigma_0 
b_0}(X_0)}\prod_{i=1}^{n}\frac{p_{\sigma 
b}(\Delta, 
X_{(i-1)\Delta},X_{i\Delta})}{p_{\sigma_0 b_0}(\Delta, 
X_{(i-1)\Delta},X_{i\Delta})} \d\nu(\sigma,b)
\le e^{-(1+c)(n\eps^2+\kappa)} \right)\\
\qquad\le \frac{6(1+\rho)}{c^2(1-\rho)(n\eps^2+\kappa)}.
\end{align*}
\end{lemma}
\begin{proof}
 By Jensen's inequality the probability in question is less than or equal to
\begin{align*}
 \PP_{\sigma_0 b_0}\Bigg(\int_{B_{\eps,\kappa}} 
\bigg(\log\frac{\mu_{\sigma b}(X_0)}{\mu_{\sigma_0 
b_0}(X_0)}+\sum_{i=1}^{n}\log\frac{p_{\sigma b}(\Delta, 
X_{(i-1)\Delta},X_{i\Delta})}{p_{\sigma_0 b_0}(\Delta, 
X_{(i-1)\Delta},X_{i\Delta})}
\bigg)&\d\nu(\sigma,b)\\
\le-(1+c)(n\eps^2+\kappa)
\Bigg).
\end{align*}
Using $\K(\mu_{\sigma_0 b_0},\mu_{\sigma b})=\E_{\sigma_0
b_0}[\log(\mu_{\sigma_0 b_0}(X_0)/\mu_{\sigma b}(X_0))]\le\kappa$, 
$\KL((\sigma_0,b_0),(\sigma,b))\le \eps^2$ for $(\sigma,b)\in 
B_{\eps,\kappa}$, Chebyshev's and again Jensen's inequality, we bound the last quantity by
\begin{align*}
& \PP_{\sigma_0 b_0}\bigg(\int_{B_{\eps,\kappa}}
\left(\log\frac{\mu_{\sigma b}(X_0)}{\mu_{\sigma_0
b_0}(X_0)}-\E_{\sigma_0 b_0}\left[\log\frac{\mu_{\sigma b}(X_0)}{
\mu_{\sigma_0
b_0}(X_0)}\right]\right.\\
&\qquad\left.\left.
+\sum_{i=1}^{n}
\left(\log\frac{p_{\sigma b}(\Delta,X_{(i-1)\Delta},X_{i\Delta})}{p_{\sigma_0
b_0}(\Delta,
X_{(i-1)\Delta},X_{i\Delta})}\right.\right.\right.\\
&\qquad\left.\left.\left.
-\E_{\sigma_0
b_0}\left[\log\frac{p_{\sigma b}(\Delta,X_{(i-1)\Delta}
,X_{ i\Delta})}{p_{\sigma_0
b_0}(\Delta,X_{(i-1)\Delta},X_{i\Delta})}\right]\right)
\right)\d\nu(\sigma,b)\right.\le-c(n\eps^2+\kappa)\bigg)\displaybreak[1]\\
&\le\frac{1}{c^2(n\eps^2+\kappa)^2}
\Var_{\sigma_0 b_0}
\left(\int_{B_{\eps,\kappa}}\log\frac{\mu_{\sigma b}(X_0)}{\mu_{\sigma_0
b_0}(X_0)}\right.\\
&\qquad\qquad\qquad\qquad\qquad\qquad\left.+\sum_{i=1}^{n}
\left(\log\frac{p_{\sigma b}(\Delta,X_{(i-1)\Delta},X_{i\Delta})}{p_{\sigma_0
b_0}(\Delta,
X_{(i-1)\Delta},X_{i\Delta})}
\right)\d\nu(\sigma,b)\right)\displaybreak[1]\\
&\le\frac{1}{c^2(n\eps^2+\kappa)^2}
\int_{B_{\eps,\kappa}}\Var_{\sigma_0 b_0}
\left(\log\frac{\mu_{\sigma b}(X_0)}{\mu_{\sigma_0 b_0}(X_0)}
\right.\\
&\qquad\qquad\qquad\qquad\qquad\qquad\left.
+\sum_{i=1}^{n}
\left(\log\frac{p_{\sigma b}(\Delta,X_{(i-1)\Delta},X_{i\Delta})}{p_{\sigma_0
b_0}(\Delta,
X_{(i-1)\Delta},X_{i\Delta})}
\right)\right)\d\nu(\sigma,b)\displaybreak[1]\\
&\le\frac{3}{c^2(n\eps^2+\kappa)^2}
\left(
\int_{B_{\eps,\kappa}}\Var_{\sigma_0
b_0}\left(\log\frac{\mu_{\sigma b}(X_0)}{\mu_{\sigma_0 b_0}(X_0)}
\right)\d\nu(\sigma,b)\right.\\
&\qquad\left.+\int_{B_{\eps,\kappa}}\Var_{\sigma_0
b_0}\left(\sum_{i=1, i\text{ odd}}^{n}
\left(\log\frac{p_{\sigma b}(\Delta,X_{(i-1)\Delta},X_{i\Delta})}{p_{\sigma_0
b_0}(\Delta,
X_{(i-1)\Delta},X_{i\Delta})}
\right)\right)\d\nu(\sigma,b)\right.\\
&\qquad\left.+\int_{B_{\eps,\kappa}}\Var_{\sigma_0
b_0}\left(\sum_{i=1, i\text{ even}}^{n}
\left(\log\frac{p_{\sigma b}(\Delta,X_{(i-1)\Delta},X_{i\Delta})}{p_{\sigma_0
b_0}(\Delta,
X_{(i-1)\Delta},X_{i\Delta})}
\right)\right)\d\nu(\sigma,b)\right)\\
&\le\frac{3}{c^2(n\eps^2+\kappa)^2}
\left(2\kappa+\frac{1+\rho}{1-\rho}2
n\eps^2\right)\le\frac{6(1+\rho)}{
c^2(1-\rho) } \cdot\frac{1}{(n\eps^2+\kappa)},
\end{align*}
where we have used Lemma~\ref{contractlem} and the 
analogue of the bound \eqref{varbound} for bivariate Markov chains 
$(X_{(i-1)\Delta},X_{i\Delta})_{i\in I}$ with either odd $I=\{1,3,\dots\}$ or
even $I=\{2,4,\dots\}$ index sets.
\end{proof}

\begin{proof}[Proof of Theorem~\ref{thm:contract}.]
 First, recalling the notation $\vartheta = (\sigma, b)$, 
$\vartheta_0=(\sigma_0, b_0)$,
\[
 \E_{\theta_0}[\Pi(\{\theta\in\Theta:d_n(\theta,\theta_0)\ge 
M\eps_n|X_0,\dots,X_{n\Delta}\})\Psi_n]\le\E_{\theta_0}[\Psi_n]\to0
\]
by assumption on the tests, so we only need to prove convergence in 
$\PP_{\theta_0}$-probability to zero of
\begin{align*}
&\quad \Pi(\{\theta\in\Theta:d_n(\theta,\theta_0)\ge 
M\eps_n|X_0,\dots,X_{n\Delta}\})(1-\Psi_n)\\
&=
\frac{\int_{d_n(\theta,\theta_0)\ge M\eps_n}
(\mu_{\theta}/\mu_{\theta_0})(X_0)\prod_{i=1}^{n}(p_{\theta}/p_{\theta_0})(\Delta , 
X_{(i-1)\Delta},X_{i\Delta})\d \Pi(\theta)}
{\int_{\Theta}
(\mu_{\theta}/\mu_{\theta_0})(X_0)\prod_{i=1}^{n}(p_{\theta}/p_{\theta_0})(\Delta , 
X_{(i-1)\Delta},X_{i\Delta})\d \Pi(\theta)}(1-\Psi_n).
\end{align*}
Lemma~\ref{anlem} shows that for the chosen $\kappa>0$, 
for all $n \in \mathbb N$, all $c>0$ and all probability measures $\nu$ with 
support in $B_{\eps_n,\kappa}$ one has 
\begin{align*}
&\PP_{\theta_0}\left(\int\frac{\mu_{\theta}}{\mu_{\theta_0}}(X_0)\prod_{i=1
} ^ { n}\frac{p_{\theta}}
{ p_ { \theta_0}}(\Delta,X_{(i-1)\Delta},X_{i\Delta})\d \nu(\theta) \le 
e^{-(1+c)(n\eps_n^2+\kappa)
} \right)\\
& \le \frac{6(1+\rho)}{c^2(1-\rho)(n\eps_n^2+\kappa)}.
\end{align*}
Setting $c=1/2$ we have for $n$ large enough 
$(1+c)(n\eps_n^2+\kappa)\le2n \eps_n^2$ so that
\begin{align*}
\PP_{\theta_0}\left(\int\frac{\mu_{\theta}}{\mu_{\theta_0}}(X_0)\prod_{i=1}^{n
} \frac{p_{\theta}}
{ p_ { \theta_0}}(\Delta,X_{(i-1)\Delta},X_{i\Delta})\d \nu(\theta) \le 
e^{-2 n \eps_n^2
} \right)\to0
\end{align*}
as $n\to\infty$. We choose $\nu$ as the normalised restriction of~$\Pi$ 
to~$B_{\eps_n, \kappa}$ and see using the condition~\eqref{eq:beksmallball} of 
the theorem, that 
for the event
\begin{align*}
A_n:=\bigg\{\int_{B_{\eps_n, 
\kappa}}\frac{\mu_{\theta}}{\mu_{\theta_0}}(X_0)\prod_{i=1}^{n}
\frac {p_{\theta}}
{ p_ { \theta_0}}(\Delta,X_{(i-1)\Delta},X_{i\Delta})\d \Pi(\theta)\\
\ge\Pi(B_{\eps_n, \kappa})e^{-2n\eps_n^2}\ge e^{-(2+C) 
n\eps_n^2}\bigg\},
\end{align*}
we have $\PP_{\theta_0}(A_n)\to1$ as $n\to\infty$. We infer for every 
$\epsilon>0$
\begin{align*}
 &\PP_{\theta_0}\bigg(\frac{\int_{d_n(\theta,\theta_0)\ge M\eps_n}
(\mu_{\theta}/\mu_{\theta_0})(X_0)\prod_{i=1}^{n}(p_{\theta}/p_{\theta_0}
)(\Delta , 
X_{(i-1)\Delta},X_{i\Delta})\d \Pi(\theta)}
{\int_{\Theta}
(\mu_{\theta}/\mu_{\theta_0})(X_0)\prod_{i=1}^{n}(p_{\theta}/p_{\theta_0}
)(\Delta , 
X_{(i-1)\Delta},X_{i\Delta})\d \Pi(\theta)}\\
&\qquad\qquad\qquad\times(1-\Psi_n)>\epsilon\bigg) \displaybreak[0] \\
&\le
\PP_{\theta_0}(A_n^c)\\
&+\PP_{\theta_0}\bigg(e^{(2+C)n\eps_n^2}
\int_{d_n(\theta,\theta_0)\ge M\eps_n}
\frac{\mu_{\theta}}{\mu_{\theta_0}}(X_0)\prod_{i=1}^{n}\frac{p_{\theta}}{p_{
\theta_0}}(\Delta, 
X_{(i-1)\Delta},X_{i\Delta})\d \Pi(\theta)\\
&\qquad\qquad\qquad\times(1-\Psi_n)>\epsilon\bigg).
\end{align*}
Using that
\begin{align*}
 \E_{\theta_0}\left[
 (1-\Psi_n)
\frac{\mu_{\theta}}{\mu_{\theta_0}}(X_0)\prod_{i=1}^{n}\frac{p_{\theta}}{p_{
\theta_0}}(\Delta, 
X_{(i-1)\Delta},X_{i\Delta})\right]
&= \E_{\theta}[1-\Psi_n]
\end{align*}
and that $0\le1-\Psi_n\le1$, we obtain
\begin{align*}
 &\quad\E_{\theta_0}\left[(1-\Psi_n)\int_{d_n(\theta,\theta_0)
\ge M\eps_n} \frac{\mu_{\theta}}{\mu_{\theta_0}}(X_0)\prod_{i=1}^{n}\frac{p_{\theta}}{p_{
\theta_0}}(\Delta, X_{(i-1)\Delta},X_{i\Delta})\d\Pi(\theta)\right]\\
&~~~~~\le\Pi(\Theta\backslash \mathcal B_n)+\sup_{\theta\in\mathcal 
B_n:d_n(\theta,\theta_0)\ge M\eps_n}\E_{\theta}[1-\Psi_n].
\end{align*}
We combine the assumption on $\mathcal B_n$ and on the tests with Markov's 
inequality to infer for every $\epsilon>0$
\begin{align*}
 \PP_{\theta_0}\bigg((1-\Psi_n)\int_{d_n(\theta,\theta_0)
\ge M\eps_n}
\frac{\mu_{\theta}}{\mu_{\theta_0}}(X_0)\prod_{i=1}^{n}\frac{p_{\theta}}{p_{
\theta_0}}(\Delta, 
X_{(i-1)\Delta},X_{i\Delta})\d\Pi(\theta)\\
>\frac{\epsilon}{e^{(2+C)n\eps_n^2}}
\bigg) \le\frac{2\bar L}{\epsilon e^{2n\eps_n^2}}
\end{align*}
and the theorem follows by combining the previous estimates since 
$n\eps_n^2\to\infty$ as $n\to\infty$.
\end{proof}

\subsection{Proof of Lemma \ref{lem:K}}

We decompose
\begin{align*}
&\quad \int_0^1\int_0^1 (K_{\sigma b}-K_{\sigma_0 b_0})^2\d \mu_0 \d 
\mu_0\\
&\lesssim
 G_{\sigma b}^2\int_0^1\int_0^1(H_{\sigma b}-H_{\sigma_0 b_0})^2\d 
\mu_0 \d \mu_0+(G_{\sigma b}-G_{\sigma_0 b_0})^2\int_0^1\int_0^1 H_{\sigma_0 
b_0}^2\d \mu_0 \d \mu_0\\
&\quad+G_{\sigma b}^2\int_0^1\left(\frac{\mu_{\sigma 
b}-\mu_0}{\mu_0}\right)^2\d \mu_0\int_0^1\left(\int_0^1 H_{\sigma b}(x,y)\d 
\mu_0(y)\right)^2\d \mu_0(x)\\
&\quad+ G_{\sigma b}^2\int_0^1\left(\int_0^1\left(H_{\sigma b}-H_{\sigma_0 
b_0}\right)(x,y)\d\mu_0(y)\right)^2\d \mu_0(x)\\
&\quad+ (G_{\sigma b}-G_{\sigma_0 b_0})^2\int_0^1\left(\int_0^1 H_{\sigma_0 
b_0}(x,y)\d\mu_0(y)\right)^2\d\mu_0(x).
\end{align*}
By the assumptions there exists a constant $C'>0$ such that $G_{\sigma b}^2<C'$ and 
$\| H_{\sigma b}\|_\infty<C'$ holds uniformly in $(\sigma, b) \in \Theta$, and 
we see that the factors in which no 
difference appears are bounded.
The term with the difference $\mu_{\sigma b}-\mu_0$ can be bounded by a multiple of $\|\mu_{\sigma b} - \mu_0\|_{L^2}^2$, using that $\mu_0$ is bounded away from zero. 

Next, by Jensen's inequality 
\begin{align*}
\left(\int_0^1\left(H_{\sigma b}-H_{\sigma_0 
b_0}\right)(x,y)\d\mu_0(y)\right)^2\le
\int_0^1\left(H_{\sigma b}-H_{\sigma_0 
b_0}\right)^2(x,y)\d\mu_0(y)
\end{align*}
so that it remains to control the right hand side of the previous inequality 
and
$G_{\sigma b}-G_{\sigma_0 b_0}$.
We start with the difference $G_{\sigma b}-G_{\sigma_0 b_0}$. We have
\begin{align}
 \left|G_{\sigma b}-G_{\sigma_0 b_0}\right|
&\le \left|
\int_0^1\left(\frac{1}{\sigma^2(y)}-\frac{1}{\sigma^2_0(y)}\right)
\exp\left(\int_0^y\frac { 2b(v)}{\sigma^2(v)}\d v
\right)\d y
\right|\notag\\
&+
\left|
\int_0^1\frac{1}{\sigma^2_0(y)}\left(\exp\left(\int_0^y\frac{2b(v)}{\sigma^2(v)}
\d v
\right)-\exp\left(\int_0^y\frac{2b(v)}{\sigma_0^2(v)}
\d v
\right)\right)\d y
\right|\label{decompG}\\
&+
\left|
\int_0^1\frac{1}{\sigma^2_0(y)}\left(\exp\left(\int_0^y\frac{2b(v)}{
\sigma_0^2(v) }
\d v
\right)-\exp\left(\int_0^y\frac{2b_0(v)}{\sigma_0^2(v)}
\d v
\right)\right)\d y
\right|\notag
\end{align}
Since $\exp\left(\int_0^y 2b(v)/\sigma^2(v)\d v
\right)$ is a bounded variation function and in $L^1$, it is contained in a 
fixed ball of
$B^1_{1\infty}$ (uniformly in $\vartheta \in \Theta$), adapting the proof of Proposition 4.3.21 in \cite{GineNickl2015} to spaces defined on $[0,1]$ (this fact will be 
used repeatedly below). So the first term can be bounded up to a 
constant by 
$\left\|\sigma^{-2}-\sigma_0^{-2}\right\|_{(B^{1}_{1\infty})^*}
$.

In order to deal with the exponential function we first show that for $f,f_0$ 
bounded and of bounded variation
\begin{align}\label{expBs}
 e^f-e^{f_0}=h(f-f_0)
\end{align}
for some $h$ bounded and of bounded variation, where bounds on $f,f_0$ and 
their bounded variation imply bounds on $h$ and its bounded variation.
Indeed
\begin{align*} 
e^{f(x)}-e^{f_0(x)}
&=e^{f_0(x)}\left(\sum_{k=1}^{\infty}\frac{(f(x)-f_0(x))^k}{k!}\right)=h(x)(f(x)-f_0(x)),
\end{align*}
where
\[ 
h(x)=e^{f_0(x)}\left(\sum_{k=1}^{\infty}\frac{(f(x)-f_0(x))^{k-1}}{k!} \right), ~~~x \in [0,1].
\]
We observe that $e^{f_0}$ is bounded and of bounded variation. Further if
$f-f_0$ and its bounded variation are bounded by $B$ then $|f-f_0|^k\le B^k$ 
and the variation of $(f-f_0)^k$ is bounded by $kB^k$.
So the sum converges to a bounded
function of bounded variation and we obtain that $h$ is bounded and of bounded 
variation.

Next we consider the second term in \eqref{decompG}
\[
 \left|
\int_0^1\frac{1}{\sigma^2_0(y)}\left(\exp\left(\int_0^y\frac{2b(v)}{\sigma^2(v)}
\d v
\right)-\exp\left(\int_0^y\frac{2b(v)}{\sigma_0^2(v)}
\d v
\right)\right)\d y
\right|.
\]
We apply \eqref{expBs} with bounded functions of bounded variation
\begin{equation}\label{fsec}
 f(y)=\int_0^y\frac{2b(v)}{\sigma^2(v)}\d v \quad\text{ and }\quad
f_0(y)=\int_0^y\frac{2b(v)}{\sigma_0^2(v)}\d v
\end{equation}
so that we can 
rewrite the second term in \eqref{decompG} as 
\begin{align*}
 &\quad\left|
\int_0^1\frac{1}{\sigma^2_0(y)}h(y)\int_0^y2b(v)\left(\frac{1}{\sigma^2(v)}
-\frac{1}{\sigma_0^2(v)}\right)
\d v
\d y
\right|\\
&=
 \left|
\int_0^1\int_v^1\frac{1}{\sigma^2_0(y)}h(y)\d 
y\,2b(v)\left(\frac{1}{\sigma^2(v)}
-\frac{1}{\sigma_0^2(v)}\right)
\d v
\right|.
\end{align*}
By the pointwise multiplier theorem 
for Besov spaces \citep[(24) on p. 143]{triebel2010} we have
\begin{align*}
 \left\|\int_v^1\frac{h(y)}{\sigma_0^2(y)}\d y\, 2b(v) 
\right\|_{B^1_{1\infty}}&
\le C  \left\|\int_v^1\frac{h(y)}{\sigma_0^2(y)}\d 
y\right\|_{B^1_{1\infty}} 
\left\|b
\right\|_{B^1_{\infty\infty}},
\end{align*}
which is finite by the bounded variation of 
$\int_v^1 h(y)/\sigma_0^2(y)\d y$ and since $b$ is bounded Lipschitz and hence 
in $B^1_{\infty\infty}$ (cf.~(4.78) in \cite{GineNickl2015}). So the second term is bounded up to a 
constant by $\left\|\sigma^{-2}-\sigma^{-2}_0\right\|_{(B^{1}_{1\infty})^*}$.

Let us consider the third term
\[
\left|\int_0^1\frac{1}{\sigma^2_0(y)}\left(\exp\left(\int_0^y\frac{2b(v)}{
\sigma_0^2(v) }
\d v
\right)-\exp\left(\int_0^y\frac{2b_0(v)}{\sigma_0^2(v)}
\d v
\right)\right)\d y\right|.
\]
We apply again \eqref{expBs} with $f(y):=\int_0^y\frac{2b(v)}{
\sigma_0^2(v) }\d v$ and $f_0(y):=\int_0^y\frac{2b_0(v)}{\sigma_0^2(v)}
\d v$, which are bounded and of bounded variation. Using 
 Fubini's theorem we see
that the third term equals
\begin{align*}
&\quad\left|\int_0^1\frac1{\sigma_0^2(y)}h(y)\int_0^y\frac{2(b(v)-b_0(v))}{
\sigma_0^2(v) } \d 
v \d y\right|\\
&=\left|\int_0^1\int_v^1\frac{h(y)}{\sigma_0^2(y)}\d y \frac2{\sigma_0^2(v)} 
(b(v)-b_0(v))\d v\right|.
\end{align*}
Again by the pointwise multiplier theorem 
for Besov spaces \citep[(24) on p. 143]{triebel2010} we have
\begin{align*}
 \left\|\int_v^1\frac{h(y)}{\sigma_0^2(y)}\d y \frac2{\sigma_0^2(v)} 
\right\|_{B^2_{1\infty}}
&
\le C  \left\|\int_v^1\frac{h(y)}{\sigma_0^2(y)}\d y\right\|_{B^2_{1\infty}} 
\left\|\frac2{\sigma_0^2(v)} 
\right\|_{B^2_{\infty\infty}},
\end{align*}
which is finite since $h/\sigma_0^2$ is of bounded variation and thus in 
$B^1_{1\infty}$ so that its primitive is in $B^2_{1\infty}$, and since
$\sigma^{-2}_0$ is $C^1$ with first derivative bounded Lipschitz, and hence 
contained in $B^2_{\infty\infty}$ (using (4.76) and (4.78) in \cite{GineNickl2015}). Consequently the third term 
is bounded up to 
a constant by $\left\|b-b_0\right\|_{(B^{2}_{1\infty})^*}$.
Summarising we have
\begin{equation}\label{diffG}
 |G_{\sigma b}-G_{\sigma_0 b_0}|\lesssim  
\left\|\frac{1}{\sigma^2}-\frac{1}{\sigma_0^2}\right\|_{(B^{1}_{1\infty})^*}+ 
\left\|b-b_0\right\|_{(B^{2}_{1\infty})^*}.
\end{equation}

It remains to show
\begin{align*}
&\int_0^1\left(H_{\sigma b}-H_{\sigma_0 
b_0}\right)^2(x,y)\d\mu_0(y)\\
&\lesssim  \|\mu_{\sigma b}-\mu_0\|^2_{L^2([0,1])}
+\left\|\frac{1}{\sigma^2}-\frac{1}{\sigma_0^2}\right\|_{(B^{1}_{1\infty})^*}^2
+ \left\|b-b_0\right\|_{(B^{2}_{1\infty})^*}^2.
\end{align*}
We define
\[
 \tilde H_{\sigma 
b}(x,z)=\int_0^1\left(\1_{[z,x]}(y)-\1_{[z,1]}(y)\int_y^1\d\mu_0 
\right)\exp\left(-\int_0^y\frac{2b(v)}{\sigma^2(v)}\d v\right)\d y
\]
so that we have $H_{\sigma b}(x,z)=\tilde H_{\sigma b}(x,z) 
\mu_{\sigma b}(z)/\mu_0(z)$.
We decompose
\begin{align*}
\int_0^1\left(H_{\sigma b}-H_{\sigma_0 
b_0}\right)^2(x,z)\d\mu_0(z)& \le 2\int_0^1\tilde H_{\sigma 
b}^2(x,z)\left(\frac{\mu_{\sigma b}(z)-\mu_0(z)}{\mu_0(z)}
\right)^2\d\mu_0(z)\\
&~~+2\int_0^1\left(\tilde H_{\sigma b}-\tilde H_{\sigma_0 
b_0}\right)^2(x,z)\d\mu_0(z).
\end{align*}
The first term in the sum can be bounded by a multiple of 
$\|\mu_{\sigma b}-\mu_0\|^2_{L^2([0,1])}$ and we focus now on the second term.

We further decompose $$(\tilde H_{\sigma b}-\tilde H_{\sigma_0 
b_0})(x,z)=(\tilde 
H_{\sigma 
b}-\tilde H_{\sigma_0 b})(x,z)+(\tilde H_{\sigma_0 b}-\tilde H_{\sigma_0 
b_0})(x,z).$$
Using \eqref{expBs} with $f(y)=-\int_0^y\frac{2b(v)}{\sigma^2(v)}\d v$,
$f_0(y)=-\int_0^y\frac{2b(v)}{\sigma_0^2(v)}\d v$ and 
corresponding $h$, a bounded function of bounded variation, we have 
\begin{align*}
& \quad (\tilde H_{\sigma b}-\tilde H_{\sigma_0 b})(x,z) \\ 
&=\int_0^1\Big(\1_{[z,x]}(y)-\1_{[z,1]}(y)\int_y^1 \d\mu_0
\Big)h(y) \int_0^y 2b(v) \Big(\frac{1}{\sigma_0^2(v)}
-\frac{1}{\sigma^2(v)}\Big)\d v\d y \\
&=\int_0^1\int_v^1\Big(\1_{[z,x]}(y)-\1_{[z,1]}(y)\int_y^1d\mu_0
\Big)h(y)\d y 2b(v)\Big(\frac{1}{\sigma_0^2(v)}-\frac{1}{\sigma^2(v)}\Big)\d v.
\end{align*}
The function 
\begin{equation} \label{funkt}
y\mapsto \left(\1_{[z,x]}(y)-\1_{[z,1]}(y)\int_y^1\d\mu_0 \right)h(y)
\end{equation}
is 
of bounded variation and thus contained in 
$B^1_{1\infty}$. Its primitive is contained in~$B^2_{1\infty}$. As above,
$b\in B^1_{\infty\infty}$.  By 
the pointwise multiplier theorem we have that
$$v \mapsto \int_v^1 \left(\1_{[z,x]}(y)-\1_{[z,1]}(y)\int_y^1\d\mu_0 
\right)h(y)\d y 2b(v)$$ is contained in $B^1_{1\infty}$ and the 
whole expression is bounded up to a constant by 
$\|\sigma^{-2}-\sigma_0^{-2}\|_{(B^1_{1\infty})^*}$.

Using again \eqref{expBs} with $f(y)=-\int_0^y\frac{2b(v)}{\sigma_0^2(v)}\d v$,
$f_0(y)=-\int_0^y\frac{2b_0(v)}{\sigma_0^2(v)}\d v$ and the 
corresponding function~$h$, we have 
\begin{align*}
&\quad (\tilde H_{\sigma_0 b}-\tilde H_{\sigma_0 b_0})(x,z)\\
&=\int_0^1\left(\1_{[z,x]}(y)-\1_{[z,1]}(y)\int_y^1\d\mu_0 
\right)h(y)\int_0^y \frac{2}{\sigma_0^2(v)}
\left(
b_0(v)
-b(v)\right)\d v\d y\\
&=\int_0^1\int_v^1\left(\1_{[z,x]}(y)-\1_{[z,1]}(y)\int_y^1\d\mu_0 
\right)h(y)\d y\, \frac{2}{\sigma_0^2(v)}
\left(
b_0(v)
-b(v)\right)\d v.
\end{align*}
As before the primitive with respect to $y$ of (\ref{funkt}) is contained in~$B^2_{1\infty}$. 
Furthermore, 
$\sigma_0^{-2}$ is $C^1$ with first derivative bounded Lipschitz, and hence 
contained in $B^2_{\infty\infty}$. Using again the pointwise multiplier theorem 
we see that 
the expression can be bounded by a constant times 
$\|b_0-b\|_{(B^2_{1\infty})^*}$.

Combining the previous two paragraphs we conclude that
\[
(\tilde H_{\sigma b}-\tilde H_{\sigma_0 b_0})(x,z)\lesssim  
\left\|\frac{1}{\sigma^2}-\frac{1}{\sigma_0^2}\right\|_{(B^{1}_{1\infty})^*}+ 
\left\|b-b_0\right\|_{(B^{2}_{1\infty})^*}
\]
and all terms in the decomposition from the beginning of the proof are treated.

\subsection{Remaining technical results}

\subsubsection{Proof of Theorem \ref{coninequmulti}}

We split the sum into two sums consisting of odd and even indices $j$, and 
prove concentration
inequalities for the two sums separately. The concentration inequality in
theorem follows then by combining the separate concentration inequalities and a
modification of the constant $\kappa$. 

Without loss of generality we only consider even indices, that is, the Markov 
chain
$(X_{2j\Delta},X_{(2j+1)\Delta})_{j\in\N}$ and verify
that it satisfies the assumptions (A\ref{mino})-(A\ref{aperi}). Let
$C=[0,1]^2$, $\nu$ the Lebesgue measure on $[0,1]^2$ and $\tilde\beta$ the
uniform lower bound on the transition probabilities of the original Markov
chain, squared. Then (A\ref{mino}) is satisfied. (A\ref{drift}) is
satisfied with $V$ constant to one and $K=1$. (A\ref{aperi}) is fulfilled
with $\beta=\tilde\beta$. Again using the results in \cite{Baxendale2005} we see that $\tau$ in Theorem~6
in \cite{Adamczak2008} is finite. By Lemma~\ref{contractlem} we have the
contraction property for the Markov
chain $(X_{2j\Delta},X_{(2j+1)\Delta} :j\in\N)$. By
the contraction property the asymptotic variance is bounded by a multiple of
$\|f\|^2_{L^2(\mu_2)}$ and yields the concentration inequality for
$(X_{2j\Delta},X_{(2j+1)\Delta}: j\in\N)$ by
Theorem~6
in \cite{Adamczak2008}.

\subsubsection{Proof of Theorem \ref{thm:tests}}
We define, for $L'>0$ to be chosen
\begin{align*}
 \Psi_n=\left\{
\begin{array}{ll}
 0 & \text{ if }d_n(\hat\theta,\theta_0)< L' \eps_n\\
 1 & \text{ if }d_n(\hat\theta,\theta_0)\ge L' \eps_n.
\end{array}
\right.
\end{align*}
It is a direct consequence of Theorem~\ref{thm:conest} that we have, for $L'$ 
large enough,
\(\E_{\theta_0}\left[\Psi_n\right] \to0\) as $n\to\infty$.
For the error of second type we obtain, for $M$ large enough depending on $L', C$ 
that, again by Theorem~\ref{thm:conest},
\begin{align*}
& \sup_{\vartheta \in \Theta_s: d_n(\theta,\theta_0)\ge M\eps_n}\E_{\theta}\left[1-\Psi_n\right] \\
&=\sup_{\vartheta \in \Theta_s: d_n(\theta,\theta_0)\ge M\eps_n}\PP_\theta\left(d_n(\hat\theta , 
\theta_0)< L' \eps_n\right)\\
&\le\sup_{\vartheta \in \Theta_s: d_n(\theta,\theta_0)\ge M\eps_n}\PP_\theta\left(d_n(\theta_0 , 
\theta)-d_n(\theta , 
\hat\theta)< L' \eps_n\right)\\
&\le \sup_{\vartheta \in \Theta_s}\PP_\theta\left(d_n(\theta, \hat\theta) > 
(M/2)\eps_n \right) \le e^{-(C+4) n \eps_n^2}.
\end{align*}

\subsubsection{Proof of Lemma \ref{besov}}

For Part b) see Proposition~4.3.13 a) in \cite{GineNickl2015}.  

For Part a), we 
can write
\begin{align*}
&\|e^f-e^g\|_{(B^1_{1\infty})^*}  = \sup_{h:\|h\|_{B^1_{1\infty}} \le 1} \left| 
\int h (e^f-e^g)\right| \displaybreak[0] \\
& = \sup_{h:\|h\|_{B^1_{1\infty}} \le 1} \left|\int h e^g (f-g) 
\left(\sum_{k=1}^\infty \frac{(f-g)^{k-1}}{k!} \right) \right|  \lesssim \|f-g\|_{(B^1_{1\infty})^*}
\end{align*}
since the $B^1_{1\infty}$-norm of $h e^g \sum_{k=1}^\infty (f-g)^{k-1}/k!$ is 
bounded by a fixed constant that depends only on $B'$, in view of the 
multiplier 
inequalities $\|h_1h_2\|_{B^1_{1\infty}} \lesssim \|h_1\|_{B^1_{1\infty}} 
\|h_2\|_{B^1_{\infty \infty}}$ and $\|h_1 h_2\|_{B^1_{\infty \infty}} \lesssim 
\|h_1\|_{B^1_{\infty \infty}} \|h_2\|_{B^1_{\infty \infty}}$ (see 
\cite{triebel2010}, p.143, (24)).

Finally, Part c) is proved as follows. Note that for $(\sigma, b) \in \Theta$ we 
have at least that $\sigma', \mu'$ are bounded Lipschitz functions and hence in 
$B^1_{\infty \infty}$. Then from (\ref{hierid}) we see that 
\begin{align*}
\|b-b_0\|_{(B^2_{1\infty})^*} &\lesssim 
\|(\sigma^2)'-(\sigma_0^2)'\|_{(B^2_{1\infty})^*} + \|(\sigma^2-\sigma_0^2) 
(\log \mu)'\|_{(B^2_{1\infty})^*} \\
&\qquad\qquad\qquad\qquad\qquad+  \|\sigma_0^2[\log \mu -\log 
\mu_0]'\|_{(B^2_{1\infty})^*} \\
&\lesssim \|\sigma^2 - \sigma_0^2\|_{(B^1_{1\infty})^*} + \|\log \mu-\log 
\mu_0\|_{L^2},
\end{align*}
where we used integration by parts and $\sigma(0)=\sigma_0(0), 
\sigma(1)=\sigma_0(1)$, similar identities for $\mu, \mu_0$ as well as the fact 
that $\|f'\|_{B^1_{1\infty}} \lesssim \|f\|_{B^2_{1\infty}}, \|fg\|_{B^\alpha_{1 
\infty}} \le \|f\|_{B^\alpha_{1\infty}} \|g\|_{B^\alpha_{\infty \infty}}$ and 
the continuous imbeddings $B^s_{pq} \subset B^{s'}_{pq}$ for all $s>s'$ and 
$B^1_{1\infty} \subset L^2$ (see \cite{triebel2010}). The final result follows by 
multiplying by $\sigma^{-2} \sigma_0^{-2} \in B^2_{\infty \infty}$ and since 
$\log$ is Lipschitz on compact subsets of $(0,\infty)$.

\subsubsection{Lipschitz properties of self-adjoint operators}

\begin{proposition}\label{prop:kitt}
 Let $N$ and $M$ be self-adjoint operators on Hilbert spaces $H_1$ and $H_2$ 
with 
scalar products $\langle\cdot,\cdot\rangle_{H_1}$ and 
$\langle\cdot,\cdot\rangle_{H_2}$, respectively. Let $f$ be a function 
defined on the union $\sigma(N)\cup\sigma(M)$ of the spectra of $N$ and $M$. If 
$|f(z)-f(w)|\le \Lambda |z-w|$ for all $z,w\in\sigma(N)\cup\sigma(M)$ and some 
positive constant~$\Lambda$, then $\|f(N)X-Xf(M)\|_{HS}\le \Lambda  
\|NX-XM\|_{HS}$ for all bounded linear operators $X$ from $H_2$ to $H_1$,
where the Hilbert--Schmidt norm is given by
$\|A\|^2_{HS}=\sum_{l} \|A \tilde e_l\|^2_{H_1}$ for $A:H_2\to 
H_1$ and an orthonormal basis $(\tilde e_l)_l$ of $H_2$. In particular, if 
$H_1$ and $H_2$ are the same with possibly different scalar products 
and 
$X=\Id:H_2\to H_1$ is bounded, then $\|f(N)-f(M)\|_{HS}\le \Lambda  
\|N-M\|_{HS}$, with $HS$-norm defined as above.
\end{proposition}
\begin{proof}
 We define the operators $D$ and $Y$ on the space $H_1\oplus H_2$ by
\begin{align*}
 D=\left(\begin{array}{cc}
N&0\\
0&M\\
\end{array}\right),\quad
 Y=\left(\begin{array}{cc}
0&X\\
0&0\\
\end{array}\right).
\end{align*}
Let $(e_k)_k$ and $(\tilde e_l)_l$ be orthonormal bases of $H_1$ and $H_2$, 
respectively. The scalar product on $H_1\oplus H_2$ is given by $\langle 
f,g\rangle_{H_1 \oplus H_2}=\langle 
f_1,g_1\rangle_{H_1}+\langle 
f_2,g_2\rangle_{H_2}$ for $f=(f_1,f_2)^\top$ and $g=(g_1,g_2)^\top$.
For $A:H_1\oplus H_2\to H_1\oplus H_2$ the Hilbert--Schmidt norm is given by
$\|A\|_{HS}=\sum_{k} \|Ae_k\|_{H_1 \oplus H_2}+\sum_{l}\|A\tilde e_l\|_{H_1 
\oplus H_2}$.
By the main theorem in \cite{Kittaneh1985} we have
\begin{align*}
 \|f(D)Y-Yf(D)\|_{HS}\le\Lambda\|DY-YD\|_{HS}.
\end{align*}
This is equivalent to 
\[
 \|f(N)X-Xf(M)\|_{HS}\le\Lambda\|NX-XM\|_{HS},
\]
where the Hilbert--Schmidt norm is for operators from $H_2$ to $H_1$.
\end{proof}

\textbf{Acknowledgement.} The authors thank an associate editor and two referees for careful reading of and critical remarks on the manuscript, Rados\l{}aw Adamczak for helpful discussions 
on concentration inequalities for Markov chains, and the 
European Research Council (ERC) for support under Grant No. 647812.

\bibliography{bib}


\end{document}